\documentclass{amsart}

\usepackage{amssymb,amsmath,amscd,textcomp,mathtools}
\usepackage{graphicx}
\usepackage[all]{xypic}
\usepackage{epsf} 
\usepackage{titlepic}
\usepackage{hyperref}
\usepackage{tabulary}
\usepackage{booktabs}

\hypersetup{
  colorlinks   = true, 
  urlcolor     = blue, 
  linkcolor    = blue, 
  citecolor   = blue 
}

\newtheorem{theorem}{Theorem}[section]
\newtheorem{lemma}[theorem]{Lemma}
\newtheorem{corollary}[theorem]{Corollary}
\newtheorem{definition}[theorem]{Definition}

\newtheorem{example}[theorem]{Example}
\newtheorem{proposition}[theorem]{Proposition}

\xyoption{arrow}
\xyoption{matrix}

\def\C{\mathbb{C}}
\def\R{\mathbb{R}}
\def\Z{\mathbb{Z}}

\def\trop{\mathbb{T}}

\def\Z{\mathbb{Z}}

\def\tree{\mathcal{T}}

\def\q{/\!/}
\def\ql{\backslash \! \backslash}

\title[Toric geometry from outer space]{Toric geometry of $SL_2(\C)$ free group character varieties from outer space}
\author{Christopher Manon}
\thanks{This work was supported by the NSF fellowship DMS-1500966.}

\begin{document}

\begin{abstract}
Culler and Vogtmann defined a simplicial space $O(g)$ called outer space to study the outer automorphism group
of the free group $F_g$.   Using representation theoretic methods, we give an embedding of $O(g)$ into the analytification of $\mathcal{X}(F_g, SL_2(\mathbb{C})),$ the $SL_2(\mathbb{C})$ character variety of $F_g,$ reproving a result of Morgan and Shalen.  Then we show that every point $v$ contained in a maximal cell of $O(g)$ defines a flat degeneration of $\mathcal{X}(F_g, SL_2(\mathbb{C}))$ to a toric variety $X(P_{\Gamma})$.  We relate $\mathcal{X}(F_g, SL_2(\mathbb{C}))$ and $X(v)$ topologically by showing that there is a surjective, continuous, proper map $\Xi_v: \mathcal{X}(F_g, SL_2(\mathbb{C})) \to X(v)$.  We then show that this map is a symplectomorphism on a dense, open subset of $\mathcal{X}(F_g, SL_2(\mathbb{C}))$ with respect to natural symplectic structures on $\mathcal{X}(F_g, SL_2(\mathbb{C}))$ and $X(v)$.  In this way, we construct an integrable Hamiltonian system in $\mathcal{X}(F_g, SL_2(\mathbb{C}))$ for each point in a maximal cell of $O(g)$, and we show that each $v$ defines a topological decomposition of $\mathcal{X}(F_g, SL_2(\mathbb{C}))$ derived from the decomposition of $X(v)$ by its torus orbits.  Finally, we show that the valuations coming from the closure of a maximal cell in $O(g)$ all arise as divisorial valuations built from an associated projective compactification of $\mathcal{X}(F_g, SL_2(\mathbb{C})).$  
\end{abstract}

\maketitle

\section{Introduction}

In their seminal paper \cite{MS}, Morgan and Shalen introduce a piecewise-linear compactification construction for any complex affine variety
$V$, where the points at infinity correspond to valuations on the coordinate ring of $V$.  By applying this construction to the $SL_2(\C)$ character varieties of surface fundamental groups, Morgan and Shalen are able to produce Thurston's piecewise linear compactification of Teichm\"uller space \cite{T}.  As part of their program, they show that isometric actions of the free group $F_g$ on metric trees are limit points of the character variety $\mathcal{X}(F_g, SL_2(\C))$, and consequently can be understood as equivalence classes of valuations on the coordinate ring $\C[\mathcal{X}(F_g, SL_2(\C))]$ (Allesandrini \cite{A} also has an account of this construction).    We explore an alternative construction of such a complex of valuations which uses a compactification of $SL_2(\C)$ and Vinberg's enveloping monoid as a centerpiece (see \cite{Vi1}, \cite{Vi2}, \cite{HMM2}, and Section \ref{sl2}).  Built into this method is a way to directly relate the symplectic geometry of $\mathcal{X}(F_g, SL_2(\C))$ to that of the flat degeneration $X(P_{\Gamma})$ associated to any valuation we construct.  Generally the flat degenerations $X(P_{\Gamma})$ are toric, so we are able to create a dense open integrable system with globally defined continuous momentum maps in $\mathcal{X}(F_g, SL_2(\C))$ for each maximal cell in the complex of valuations we consider (this construction should be compared to the main result of Harada and Kaveh in \cite{HK}).  In short, from the perspectives of both algebraic and symplectic geometry, $\mathcal{X}(F_g, SL_2(\C))$ is ``almost'' a toric variety in many ways.

Our first main result is a reformulation of Morgan and Shalen's construction of a valuation on $\C[\mathcal{X}(F_g, SL_2(\C))]$ from an isometric
$F_g$ action on a metric tree; in order to state it we recall spaces from geometric group theory and the theory of Berkovich analytification.  In \cite{CV}, Culler and Vogtmann introduce a space $O(g)$ with an action of the outer automorphism group $Out(F_g)$ called \emph{outer space}.   Points of $O(g)$ are graphs $\Gamma$ with no leaves and first Betti number equal to $g$, along with choice of metric $\ell$ and a ``marking'' (see Section \ref{outerspace}), which provides an isomorphism $\phi: \pi(\Gamma) \cong F_g$ (see \cite{CV}, \cite{V},  Section \ref{outerspace}).  The metric is subject to a normalization condition, stipulating that the sum of the lengths of the edges of $\Gamma$ must be $1.$   Culler and Vogtmann show that $O(g)$ has a simplicial structure, with a simplicial action by $Out(F_g)$ and use this to prove that $Out(F_g)$ has a number of group theoretic properties (see \cite{V} for a survey of these results).   Points in $O(g)$ can be viewed as isometric actions of $F_g$ on trees by replacing $(\Gamma, \ell)$ with its universal cover, a metric tree with an isometric action of $F_g$ defined by the isomorphism $\phi$, this places $O(g)$ in the context studied by Morgan and Shalen in \cite{MS}.  We drop the normalization condition on the metric and consider the resulting complex of simplicial cones, $\hat{O}(g)$, referred to as the cone over outer space. 

  Each reduced word $\omega \in F_g$ defines a continuous function $d_{\omega}: \hat{O}(g) \to \R$, where $d_{\omega}(\Gamma, \ell, \phi)$ is the length of the minimal length path in $(\Gamma, \ell)$ which represents $\omega$ in $\pi_1(\Gamma)$. Reduced words $\omega \in F_g$ also define regular functions $tr_{\omega} \in \C[\mathcal{X}(F_g, SL_2(\C))]$ called trace-word functions, see Sections \ref{outerspace} and \ref{spanningsets}.   A point in $\mathcal{X}(F_g, SL_2(\C))$ is an $SL_2(\C)$ representation of $F_g$, this is a choice of $g$ matrices $A_1, \ldots, A_g \in SL_2(\C)$. Two choices of matrices define the same representation if and only if they are related by simultaneous conjugation by an element of $SL_2(\C)$. More precisely, $\mathcal{X}(F_g, SL_2(\C))$ is defined as a Geometric Invariant Theory ($GIT$) quotient. 

\begin{equation}
\mathcal{X}(F_g, SL_2(\C)) = [SL_2(\C) \times \ldots \times SL_2(\C)] \q SL_2(\C)\\
\end{equation}

\noindent
The function $tr_{\omega}$ is computed on $[A_1, \ldots, A_g] \in \mathcal{X}(F_g, SL_2(\C))$ by evaluating $\omega(A_1, \ldots, A_g)$
and taking the trace of the resulting matrix. 

The Berkovich analytification $V^{an}$ of an affine variety $V$ (see \cite{Ber}, \cite{P},  Section \ref{valuations}) is a Hausdorff topological space composed of all the rank $1$ valuations on the coordinate ring of $V$.  Every regular function $f \in \C[V]$ defines a function on the analytification $ev_f: V^{an} \to \R$ called the evaluation function, it is computed by taking a valuation $v$ to $v(f)$. The topology on $V^{an}$ is the weakest one which makes the evaluation functions continuous.

\begin{theorem}\label{mainvaluation}
There is an embedding $\Sigma: \hat{O}(g) \to \mathcal{X}(F_g, SL_2(\C))^{an}$, furthermore, for any reduced word $\omega \in F_g,$ $ev_{tr_{\omega}}\circ \Sigma = d_{\omega}.$
\end{theorem}

 Theorem \ref{mainvaluation} is analogous to Theorem $3.4$ in \cite{SpSt}, which identifies the tropical variety of the Grassmannian variety $Gr_2(\C^n)$ with the space of phylogenetic trees with $n$ ordered leaves. In \cite[Proposition 3.1]{M5}, the author uses combinatorial and representation theoretic methods to construct an embedding of the space of metric trees into the analytification $Gr_2(\C^n)^{an}$. We use similar methods to give a new proof of Theorem \ref{mainvaluation}, employing a theory of valuations stemming from compactifications of $SL_2(\C)$, Vinberg's remarkable enveloping monoid construction, and a noncommutative quiver variety construction.  

In \cite{FL}, Florentino and Lawton study varieties $M_{\Gamma}(G)$ for $G$ a reductive group, built as noncommutative
analogues of quiver varieties.  They show that each $M_{\Gamma}(G)$ for $\Gamma$ a directed graph with first Betti number equal to $g$ is isomorphic to $\mathcal{X}(F_g,G)$ (see also \cite{M15}). We extend this construction to a functor $M_{-}(SL_2(\C))$ from the category of graphs with certain finite cellular topological maps to the category of complex schemes.  We show that an isomorphism $\phi: \pi_1(\Gamma) \to F_g$ defines an isomorphism $M_{\Gamma}(SL_2(\C)) \cong \mathcal{X}(F_g, SL_2(\C))$, and that the variety $M_{\Gamma}(SL_2(\C))$ comes with a naturally defined simplicial cone $C_{\Gamma}$ of valuations on its coordinate ring. The pullbacks $C_{\Gamma, \phi}$ of these cones in $\mathcal{X}(F_g, SL_2(\C))^{an}$ paste together to give the embedding $\Sigma.$

The graph $\Gamma$ also plays a fundamental role in our analysis of the coordinate ring of $\mathcal{X}(F_g, SL_2(\C)).$  For a trivalent directed
graph $\Gamma$, a spin diagram of topology $\Gamma$ is an assignment of non-negative integers $a: E(\Gamma) \to \Z_{\geq 0}$
to the edges of $\Gamma$ which satisfy the inequalities of a polyhedral cone $\mathcal{P}_{\Gamma}$ and belong to a certain lattice $L_{\Gamma} \subset \R^{E(\Gamma)}$ (see Figure \ref{spin} and Section \ref{spanningsets}).

\begin{figure}[htbp]
\centering
\includegraphics[scale = 0.4]{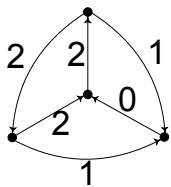}
\caption{A spin diagram.}
\label{spin}
\end{figure}

There is a basis of regular functions $\Phi_a, a \in P_{\Gamma} = \mathcal{P}_{\Gamma} \cap L_{\Gamma}$ 
in $\C[\mathcal{X}(F_g, SL_2(\C))]$ for each marked $\Gamma$ which does not depend on the directed structure on $\Gamma$ (see Section \ref{spanningsets}).   These ``spin diagram'' functions have a rich combinatorial theory, and have been studied before by Lawton and Peterson, \cite{LP}.  Spin diagrams are a beautiful combinatorial tool in topological field theory and knot theory, see \cite{Baez} for a survey of some of these topics. 

The set $P_{\Gamma}$ is an affine semigroup with associated affine semigroup algebra $\C[P_{\Gamma}]$.  For $a, b \in P_{\Gamma}$ the product $[a][b] \in \C[P_{\Gamma}]$ is computed naively, $[a][b] = [a + b]$. We let $X(P_{\Gamma})$  be the  affine toric variety $Spec(\C[P_{\Gamma}])$. 
Our second main theorem says that the associated graded multiplication operation in $\C[\mathcal{X}(F_g, SL_2(\C))]$  with respect to the filtration
defined by the valuation $\Sigma(\Gamma, \ell, \phi) = v_{\Gamma, \ell, \phi}$ is this naive multiplication operation. 

\begin{theorem}\label{mainassociatedgraded}
The associated graded algebra $gr_{v_{\Gamma, \ell, \phi}}(\C[\mathcal{X}(F_g, SL_2(\C))])$ is isomorphic to $\C[P_{\Gamma}]$.
\end{theorem}

Let $S_{2, g} \subset \C[\mathcal{X}(F_g, SL_2(\C))]$ be the set of traceword functions $\tau_w$ for reduced words $w$ in which any letter appears at most twice, counting inverses and multiplicity. In Section \ref{spanningsets} we use Theorem \ref{mainassociatedgraded} to show that the functions $S_{2, g}$ generate  $\C[\mathcal{X}(F_g, SL_2(\C))]$ and that the tropical variety $\trop(I_{2, g})$ of the ideal of forms which vanish on $S_{2, g}$
contains the polyhedral complex $\Upsilon_g$ of ``metric spanned graphs'' with first Betti number $g$ (see Subsection \ref{spanneddef} and Theorem \ref{Khovanskii}).

Next we analyze a relationship between $\mathcal{X}(F_g, SL_2(\C))$ and $X(P_{\Gamma})$ in symplectic geometry. We relate $\mathcal{X}(F_g, SL_2(\C))$ to $X(P_{\Gamma})$ topologically by constructing a continuous map $\Xi_{\Gamma}: \mathcal{X}(F_g, SL_2(\C)) \to X(P_{\Gamma})$ in Section \ref{Hamiltonian}. We recall that as a toric variety, $X(P_{\Gamma})$ is stratified by toric subvarieties $X(F)$ for $F \subset P_{\Gamma}$ the affine semigroups of $\mathcal{L}_{\Gamma}$ points in a face  $\mathcal{F} \subset \mathcal{P}_{\Gamma}$."

\begin{theorem}\label{maintoric}
The map $\Xi_{\Gamma}$ is surjective, continuous, and proper.  If a face $\mathcal{F} \subset \mathcal{P}_{\Gamma}$ is not contained
in a coordinate hyperplane of $\R^{E(\Gamma)}$, $\Xi_{\Gamma}: \Xi_{\Gamma}^{-1}(X(F)) \to X(F)$
is a homeomorphism.  If $\mathcal{F}$ is contained in a coordinate hyperplane of $\R^{E(\Gamma)}$, $\Xi_{\Gamma}: \Xi_{\Gamma}^{-1}(X(F)) \to X(F)$
has connected fibers, and if $\mathcal{F}$ is the origin of $\R^{E(\Gamma)}$, the fiber of this bundle
is the compact character variety $\mathcal{X}(F_g, SU(2))$. 
\end{theorem}

In Section \ref{Hamiltonian} we place symplectic structures on both $\mathcal{X}(F_g, SL_2(\C))$ and $X(P_{\Gamma})$, derived
from the symplectic form on $SL_2(\C)$ obtained from an identification with the cotangent bundle $T^*(SU(2)).$   

\begin{theorem}\label{mainsymplectic}
The map $\Xi_{\Gamma}$ is a symplectomorphism on a dense, open subset of $\mathcal{X}(F_g, SL_2(\C)).$ 
\end{theorem}

\noindent
The map $\Xi_{\Gamma}$ transfers the integrable system defined by the open torus orbit
of $X(P_{\Gamma})$ to the character variety $\mathcal{X}(F_g, SL_2(\C))$.    This result makes
use of recent work \cite{HMM2} of Hilgert, the author, and Martens on the symplectic geometry of Vinberg's enveloping monoid construction.

Our last theorem shows that the valuations $v_{\Gamma, \ell, \phi}$ come from divisorial valuations on a compactification of $\mathcal{X}(F_g, SL_2(\C)).$   We bring in a compactification $X$ of the group $SL_2(\C)$ and use it to build 
a space $M_{\Gamma}(X)$ by carrying out the the construction $M_{\Gamma}(-)$ for the $SL_2(\C) \times SL_2(\C)$ space $X,$ 
see Section \ref{compactification}. 

\begin{theorem}\label{maincompact}
For each isomorphism $\phi: \pi_1(\Gamma) \to F_g$, there is a compactification of $\mathcal{X}(F_g, SL_2(\C))$ by the projective
scheme $M_{\Gamma}(X)$.  The boundary divisor of this compactification is of combinatorially normal crossings type, and
the divisorial valuations of its irreducible components generate the extremal rays of the cone $C_{\Gamma, \phi}.$ 
\end{theorem}

  In Section \ref{compactification} we extend the toric degenerations of Theorem \ref{mainassociatedgraded} to the compactification $M_{\Gamma}(X)$, obtaining a degeneration $M_{\Gamma}(X) \Rightarrow X(Q_{\Gamma})$ to a toric variety associated to a polytope $\mathcal{Q}_{\Gamma} \subset \mathcal{P}_{\Gamma}$.  The strata of the boundary divisor in Theorem \ref{maincompact} are shown to degenerate to toric varieties corresponding to the faces of $\mathcal{Q}_{\Gamma}.$

A degeneration $\C[\mathcal{X}(F_g, SL_2(\C))] \Rightarrow \C[P_{\Gamma}]$ was studied by the author in \cite{M14} in the
context of the theory of Newton-Okounkov bodies (see \cite{KK}).  In particular $\C[P_{\Gamma}]$ was realized as the associated graded algebra of $\C[\mathcal{X}(F_g, SL_2(\C))]$ with respect to a full rank ($=$ dimension of $ \mathcal{X}(F_g, SL_2(\C))$) valuation.  With Theorem \ref{NOK} (Section \ref{compactification}) we show that this maximal rank valuation can be defined using the boundary divisor of $M_{\Gamma}(X).$ 

Theorems \ref{maintoric} and \ref{mainsymplectic} are inspired by earlier work of the author with Howard and Millson, \cite{HMM}, and the recent work of Harada and Kaveh \cite{HK} and Nishinou, Nohara, and Ueda \cite{NNU}.  In \cite{HK}, the authors use a construction of Ruan \cite{R} to show that a degeneration of a smooth, projective variety $Y \subset \mathbb{P}^N$ to a toric variety $Y(C)\subset \mathbb{P}^N$ guarantees the existence of a surjective, continuous, proper map $\Phi: Y \to Y(C)$ which is a symplectomorphism on a dense, open subset of $Y$ with respect to the standard K\"ahler form on $\mathbb{P}^N$; this defines a dense, open integrable system in $Y$ with momentum image $C.$  Theorems \ref{mainvaluation} and \ref{mainsymplectic} combine to produce an explicitly computable map for the variety $\mathcal{X}(F_g, SL_2(\C))$, which is notably both affine and singular in general. 

  The combinatorial organization of the toric degenerations of $\mathcal{X}(F_g, SL_2(\C))$ in this paper are remniscient of the work of Gross, Hacking, Keel, and Kontsevich, \cite{GHKK} on cluster algebras.  In the spirit of their work and this paper, it would be interesting to see an explicit construction of an integrable system in a cluster variety for each choice of seed in the associated cluster algebra. 

 In \cite{M15} the author shows that the character variety $\mathcal{X}(F_g, SL_2(\C))$ is a dense open subspace in a degeneration $M_{C_{\Gamma}}(SL_2(\C))$ of the moduli space of semistable $SL_2(\C)$ principal bundles on a smooth curve.  In particular the 
projective varietes $M_{\Gamma}(SL_2(\C))$ and $M_{C_{\Gamma}}(SL_2(\C))$ are birational.   We note that Theorem \ref{NOK} and results from
\cite{M15} then imply that the toric degenerations of $M_{C_{\Gamma}}(SL_2(\C))$ studied in \cite{M15} and \cite{M4} emerge from 
the Newton-Okounkov body construction defined by the boundary divisor in $M_{\Gamma}(SL_2(\C))$, see Subsection \ref{NOKsection}.

\subsection{Acknowledgements}

We thank  Joachim Hilgert, Kiumars Kaveh, Sean Lawton, and Johan Martens for useful conversations which contributed to this project.

\tableofcontents

\section{Metric graphs and outer space}\label{outerspace}

In this section we review material on metric graphs and cellular maps, and we recall Culler and Vogtmann's (\cite{V}, \cite{CV})definition of $O(g)$ and the length functions $d_{\omega}: \hat{O}(g) \to \R.$ 

\subsection{Graph notation and $\hat{O}(g)$}\label{graphnotation}

Throughout the paper $\Gamma$ denotes a graph with edge set $E(\Gamma)$, vertex set $V(\Gamma)$, and leaf set $L(\Gamma)$.  We require all non-leaf vertices $v \in V(\Gamma)$ to have valence at least $3$.  Many of our constructions make use of an orientation on a graph, this is an ordering $\delta(e) = (u, v)$ of the endpoints of each edge $e \in E(\Gamma).$ An orientation also defines a partition $\epsilon(v) = i(v) \cup o(v)$ of the edges which contain a given vertex into incoming and outgoing edges. For any vertex $v \in V(\Gamma)$ we let the link $\Gamma_v \subset \Gamma$ be the subgraph induced by those edges in $\epsilon(v)$.     For each genus $g$ we fix a distinguished graph $\Gamma_g$ with a single vertex and no leaves.

The points of both outer space $O(g)$ and the cone over outer space $\hat{O}(g)$ are represented by marked  $\R$-graphs.  Topologically, graphs are considered as finite $CW$ complexes, in particular edges are $1$-cells, vertices are $0$-cells. An $\R$-graph structure on a graph $\Gamma$ is a metric which makes each edge locally isometric to a bounded interval in $\R$, such that the distance between any two vertices is the length of the shortest edge-path which connects them.  A continuous map $\psi: \Gamma_g \to \Gamma$ is called a marking if $\pi_1(\psi)$ is an isomorphism of fundamental groups.  Two markings $\psi_1, \psi_2: \Gamma_g \to \Gamma$ are said to be equivalent if there is an isometry $h: \Gamma \to \Gamma$ so that the diagram in Figure \ref{markingdiagram} commutes up to free homotopy (see \cite[Section 0]{CV}).  

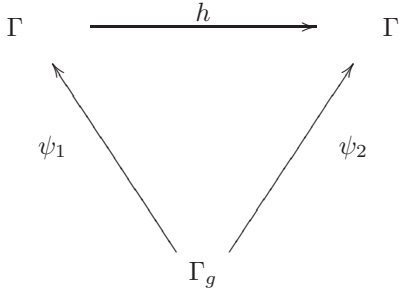
\begin{figure}[htbp]

$$
\begin{xy}
(0, -3)*{\Gamma_g} = "A";
(3.5, 0)*{} = "A1";	
(-3.5, 0)*{} = "A2";
(25, 30)*{\Gamma} = "B";
(20,25)*{} = "B1";
(15, 30)*{} = "B2";
(-25, 30)*{\Gamma} = "C";
(-20,25)*{} = "C1";
(-15, 30)*{} = "C2"; 
(0, 32)*{h};
(-20, 14)*{\psi_1};
(20, 14)*{\psi_2};
{\ar@{>}"A1"; "B1"};
{\ar@{>}"A2"; "C1"};
{\ar@{>}"C2"; "B2"};
\end{xy}
$$\\
\caption{Equivalence of markings}
\label{markingdiagram}
\end{figure}

From now on a map of graphs $\phi: \Gamma \to \Gamma'$ is taken to be a finite cellular map on the underlying topological spaces.  In particular, for any edge $e \in E(\Gamma)$ with $\delta(e) = (u, v)$, there is a finite subset of points $S \subset e$ (considered as a compact interval) such that the sub interval between two consecutive members $s,t \in S$ maps homeomorphically to an edge in $e' \in E(\Gamma')$ through $\pi$, with $s$ and $t$ mapping to the endpoints of $e'$.  Any such $\pi: \Gamma \to \Gamma'$ defines a set map on vertices $\pi_v: V(\Gamma) \to V(\Gamma')$.

Culler and Vogtmann observe that every marking can be represented with a finite cellular map $\psi: \Gamma_g \to \Gamma$ \cite[1.2]{CV} as follows.    For a spanning tree $\tree \subset \Gamma$, collapsing $\tree$ to a single vertex defines a homotopy equivalence $\psi_{\tree}: \Gamma \to \Gamma_g$.  Notice that for each $e_i \in E(\Gamma_g)$ there is a unique edge $\bar{e}_i \in E(\Gamma) \setminus E(\tree)$ which maps to $e_i$ under $\psi_{\tree}$.  An inverse $\phi_{\tree, V}$ to $\psi_{\tree}$ can be constructed from a choice of vertex $V \in V(\Gamma)$ by sending $e_i$ to the path formed by concatinate the unique path in $\tree$ from $V$ to the source endpoint of $\bar{e}_i$ with $\bar{e}_i$ and the unique path in $\tree$ from the sink point in $\bar{e}_i$ back to $V$.  Changing $V$ to a different vertex induces an equivalent inverse map on free homotopy. We say that $\phi_{\tree, V}: \Gamma_g \to \Gamma$ is a distinguished map.  Any outer automorphism $\alpha$ of $F_g$ also defines a map on $\Gamma_g$ by sending an edge $e_i$ to the cellular path in $\Gamma_g$ dictated by the image of the $i$-th generator of $F_g$ under $\alpha$. 

\begin{lemma}[Culler, Vogtmann]\label{distinguishedequivalence}
Any marking $\phi: \Gamma_g \to \Gamma$ is equivalent to a composition of a distinguished marking $\phi_{\tree, V}: \Gamma_g \to \Gamma$ with a map $\alpha: \Gamma_g \to \Gamma_g$ where $\pi_1(\alpha): F_g \to F_g$ is an outer automorphism.
\end{lemma}																

\begin{proof}
For $\phi: \Gamma_g \to \Gamma$ choose a spanning tree $\tree \subset \Gamma$ and consider $\psi_{\tree} \circ \phi: \Gamma_g \to \Gamma_g$. Let $w: F_g \to F_g$ be the inverse of $\pi_1[\psi_{\tree} \circ \phi]$, and let $\alpha_w: \Gamma_g \to \Gamma_g$ be the corresponding map.  Then for any $V \in V(\Gamma)$, $\phi_{\tree, V}\circ \alpha_w^{-1}: \Gamma_g \to \Gamma$ represents the inverse to $\alpha_w\circ \phi_{\tree}$ in homotopy, and is therefore equivalent to $\phi$.
\end{proof}

  As a set, the cone over outer space $\hat{O}(g)$ is the collection of equivalence classes of markings of a metric graphs of genus $g.$  
The volume of a metric graph is defined in \cite{CV} to be the sum of the lengths of the edges.  Outer space $O(g) \subset \hat{O}(g)$ is then the set of those marked metric graphs with volume $1$.   For a fixed graph $\Gamma$, the collection of supported metrics forms a simplicial cone $C_{\Gamma}$, equal to the positive orthant in $\R^{E(\Gamma)}.$ The points on the boundary of $C_{\Gamma}$ can be viewed as metrics on degenerations of $\Gamma$, where some edges with a $0$ length are collapsed. Each marking $\phi$ defines a cone $C_{\phi, \Gamma} \subset \hat{O}(g)$ isomorphic to $C_{\Gamma}$.  \\

\begin{figure}[htbp]
\centering
\includegraphics[scale = 0.6]{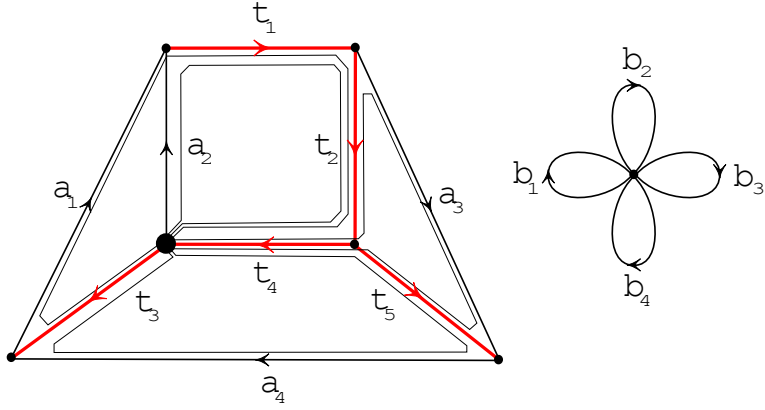}
\caption{A spanning tree $\tree$ in $\Gamma$. The associated maps $\phi, \psi$ are computed below.}
\label{collapsecocollapse}
\end{figure}

\begin{equation}
\psi(a_i) = b_i, \ \ \ \ \psi(t_i) = 1,\\
\end{equation}

$$\phi(b_1) = t_3a_1t_1t_2t_4,  \ \ \ \phi(b_2) = a_2t_1t_2t_4, \ \ \ \phi(b_3) =  t_4^{-1}t_2^{-1}a_3t_5^{-1}t_4, \ \ \ \phi(b_4) = t_4^{-1}t_5a_4t_3^{-1}$$\\

\subsection{Minimal length homotopy classes and length functions}

The cellular length of a path $e_1\cdots e_n = \gamma \subset \Gamma$ is equal to the edge count $n$ (note that this is different than the length assigned to this path by a metric).  A backtrack in a path $\gamma \subset \Gamma$ is a sequence of edges of the form $e_1 \cdots e_k e_k^{-1} \cdots e_1^{-1}$ (inverse is taken with respect to any background orientation).  When a path $\gamma$ has no backtracks, we say that it is reduced.   For any metric graph $\Gamma$ and free homotopy class $[\gamma]$, there is a unique, reduced cellular representative $\gamma \sim \gamma_{min}$ in $\Gamma.$ In particular, $\gamma_{min}$ is the unique representative of $[\gamma_{min}]$ without a backtrack. 

\begin{proposition}\label{backtrack}
Fix a graph $\Gamma$ and a class $[\alpha] \in \pi_1(\Gamma)$ then there is a unique minimal cellular length loop $\gamma$ with $[\gamma] = [\alpha]$ such that any cellular loop $\gamma'$ with $[\gamma'] = [\alpha]$ can be transformed into $\gamma$ be eliminating backtracks.  Furthermore, $\gamma$ also has the smallest metric length in its class according to any metric $\ell$ on $\Gamma$.
\end{proposition}

\begin{proof}
First we pick a spanning tree $\tree \subset \Gamma$ with the associated map $\phi_{\tree}: \Gamma \to \Gamma_g$.   Since $\pi_1[\phi_{\tree}]$ is an isomorphism, there is a reduced word $w \in F_g$ such that any $\gamma'$ as above must have $\pi_1(\phi_{\tree})[\gamma']$ equal to the homotopy class determined by $w$. In particular if $\bar{e}_i$ is the edge in $E(\Gamma) \setminus E(\tree)$ which maps to $e_i \in E(\Gamma_g)$, the $\bar{e}_i$ appear in $\gamma'$ in the same cyclic order as the $e_i$ appear in $w$, and in between the corresponding $e_i$ in $\phi_{\tree}(\gamma')$ we must have backtracks which can be eliminated.  

Now notice that if $e_i$ and $e_j$ appear consecutively in $\phi_{\tree}(\gamma')$ then the relevant piece of $\gamma'$ must be of the form $\bar{e}_i\rho\bar{e}_j$ for $\rho$ the unique path in $\tree$ which connects the endpoints of $\bar{e}_i$ and $\bar{e}_j$.  

Now by these two observations,  if $e_i$ and $e_j$ appear consecutively in the class $[\phi_{\tree}(\gamma')]$ then in $\gamma'$ we must have $\bar{e}_i\rho'\bar{e}_j$, where $\rho'$ reduces to $\rho$ after eliminating all backtracks.  Carrying out all of these reductions results in a path $\gamma$ in $\Gamma$ constructed by taking the edges $\bar{e}_i$ in $w$ along with the unique paths in $\tree$ between their consecutive endpoints.  By the observations above, this path is the unique path of shortest cellular length in $\Gamma$ which can map to $w$.

By construction, every edge which appears in $\gamma$ also appears in any $\gamma'$, it follows that the length of $\gamma$ is less than or equal to that of $\gamma'$ with respect to any metric $\ell$. 
\end{proof}

The cone over outer space has a natural set of coordinate functions defined by reduced words in the free group $F_g$.  Each element $\omega \in F_g$ defines a homotopy class in $\pi_1(\Gamma_g)$, which can then be pushed forward to $\Gamma$ by the marking $\phi$.  By passing to the unique reduced representative of $\pi_1(\phi)(\omega)$, we obtain a function $d_{\omega}: C_{\phi, \Gamma} \to \R_{\geq 0}$ by measuring the total length of this representative in $(\Gamma, \ell)$, see  \cite[page 2]{CV}, \cite[page 6]{V}.  Since this function is an invariant of the homotopy class, conjugate words define the same length function. Let $<F_g>$ denote the equivalence classes under conjugation.    The collection of length functions $d_{\omega}, \omega \in F_g$ define an embedding $\iota: \hat{O}(g) \to \R^{<F_g>}$, in particular the length function values determine the $\R$-graph structure on $\Gamma$ (see \cite{CM}). Furthermore, the cones $C_{\Gamma, \phi}$ define a decomposition of $\hat{O}(g)$ into simplicial cones in this topology, \cite{V}.  From now on a marked metric graph is a triple $(\Gamma, \phi, \ell)$ of a graph $\Gamma$, a marking $\phi$ and a metric $\ell \in C_{\Gamma, \phi}$.

\subsection{The space $\Upsilon_g$ of metric spanned graphs}\label{spanneddef}

Let $\Gamma$ be a graph with $\beta_1(\Gamma) = g$, and let $\tree \subset \Gamma$ be a spanning tree.  The complement $E(\Gamma) \setminus E(\tree)$ is a set of size $g$; we label these edges $1, \ldots, g$ and give each one an orientation.  Taken together, we say this information makes $\Gamma$ into a spanned graph.  The data of a spanned graph structure on $\Gamma$ is equivalent to giving a map $\psi_{\tree}: \Gamma \to \Gamma_g$ as in Subsection \ref{graphnotation}.  We say two spanned graphs $\Gamma, \Gamma'$ are equivalent if there is a graph isomorphism $h: \Gamma \to \Gamma'$ such that $\psi_{\tree'}\circ h = \psi_{\tree}$.  We let $\Upsilon_g$ be the set of triples $(\Gamma, \psi_{\tree}, \ell)$, where $\psi_{\tree}$ defines a spanned graph structure on $\Gamma$ and $\ell$ is a metric on $\Gamma$.  

Recall the space $\mathfrak{T}_n$ of phylogenetic trees with $n$ leaves introduced by Billera, Holmes, Vogtmann \cite{BHV}.  A point in $\mathfrak{T}_n$ is a tree $\tree$ with $n$ leaves labeled $1, \ldots, n$ and a metric represented as a function $\ell: E(\tree) \to \R_{>0}$.   For a phylogenetic tree $\tree$ we let $f_i \in E(\tree)$ be the edge connected to the $i$-th leaf. 

\begin{proposition}
The set $\Upsilon_g$ can be identified with the subset of $\mathfrak{T}_{2g}$ of phylogenetic trees with $\ell(f_{2i}) = \ell(f_{2i-1})$ for $1 \leq i \leq g$. 
\end{proposition}

\begin{proof}
Let $(\Gamma, \psi_{\tree}, \ell) \in \Upsilon_g$, and let $\{e_1, \ldots, e_g\} = E(\Gamma) \setminus E(\tree)$. Each $e_i$ can be split into two edges $f_{2i}, f_{2i-1}$, where $f_{2i}$ is connected to the vertex at the head of the orientation on $e_i$, this defines a labeled tree $\tree'$ with $2g$ leaves.  By stipulating that $\ell'(f_{2i}) = \ell'(f_{2i-1}) = \ell(e_i)$ and $\ell'(h) = \ell(h)$ for all edges in $\tree$, we define the structure of a phylogenetic tree on $\tree'$. This procedure can be reversed by replacing $f_{2i}$ and $f_{2i-1}$ with an oriented edge $e_i$, and the resulting graph is clearly isomorphic to $\Gamma$ as a spanned graph.  As this latter procedure can be performed with any point in $\mathfrak{T}_{2g}$, we have $\Upsilon_g \cong \mathfrak{T}_{2g}$. 
\end{proof}

The polyhedral complex structure on $\mathfrak{T}_{2g}$ defines the structure of a polyhedral complex on $\Upsilon_g$.  For each spanned graph $(\Gamma, \psi_{\tree})$ we choose a basepoint $V \in V(\Gamma)$; this defines a homotopy inverse $\phi_{\tree, V}$ to $\psi_{\tree}$, and a marked graph structure on $\Gamma$.  Any two choices of basepoint define the same marking by definition.  In this way, $\Upsilon_g$ can be identified with a subset of $\hat{O}_g$.

\begin{proposition}
Any point $(\Gamma, \phi, \ell) \in \hat{O}_g$ is in the image of $\Upsilon_g$ under the action of $Out(F_g)$.
\end{proposition}  

\begin{proof}
This follows from Lemma \ref{distinguishedequivalence}. 
\end{proof}

\section{Valuations, filtrations, and analytification}\label{valuations}

In this section we review technical details of valuations which we use in Section \ref{charactervaluations}. We also recall some of the structure of the analytification $X^{an}$ of a complex affine variety $X = Spec(A).$ 

\subsection{Valuations on tensor products}

Throughout the paper we us the MAX convention in the definition of a valuation. A function $v: A \to \R \cup \{-\infty\}$ on a commutative $\C$ domain is known as a rank $1$ valuation which lifts the trivial valuation on $\C$ if it satisfies:\\

\begin{enumerate}
\item $v(ab) = v(a) + v(b)$,\\
\item $v(a + b) \leq max\{v(a), v(b)\}$,\\
\item $v(C) = 0$ for all $C \in \C^*$, and\\
\item $v(0) = -\infty$.\\
\end{enumerate}

\noindent
A valuation $v$ defines a filtration on $A$ by the spaces $v_{\leq k} = \{f \in A | v(f) \leq k\} \subset A;$ we let $gr_v(A)$ be the associated graded algebra of this filtration. An increasing algebraic $\R-$filtration $F$ by complex vector spaces likewise defines a real-valued function $v_{F}: A \to \R \cup \{-\infty\}$ by sending $f$ to $inf\{r | f \in F_{\leq r}(A)\}.$ This function is a valuation as above if and only if $gr_F(A)$ is a domain. More generally, $v_F$ is said to be a quasivaluation.

Let $v$ and $w$ be valuations on commutative $\C-$domains $A$ and $B$, these can be used to define filtrations on the tensor product $A \otimes B$ (taken over $\C$):

\begin{equation}
v_{\leq k}(A \otimes B) = v_{\leq k}(A) \otimes B \ \ \ \ w_{\leq k}(A \otimes B) = A \otimes w_{\leq k}(B).\\
\end{equation}

The sum $v \oplus w$ of two valuations $v, w$ is the quasivaluation corresponding to the filtration defined by the following subspaces: 
 
\begin{equation}
(v \oplus w)_{\leq m} = \sum_{i + j \leq m} v_{\leq i} \cap w_{\leq j}.\\
\end{equation}

\noindent
Similarly, for $R \in \R_{\geq 0},$ the multiple $R\circ v$ of a valuation is defined by setting

\begin{equation}
(R\circ v)_{\leq m} = \sum_{k \leq \frac{m}{R}} v_k(A)\\
\end{equation}
 
\begin{lemma}\label{val2}
Let $v, w: A \otimes B \to \R \cup \{-\infty\}$ be as above. Any linear combination $S\circ v \oplus T\circ w$, $S, T \in \R_{\geq 0}$ 
defines a valuation on $A \otimes B.$
\end{lemma}

\begin{proof}
The associated graded algebra of $v$ as a valuation on $A \otimes B$ is $gr_v(A) \otimes B$, which is manifestly a domain.  This implies that $v$ and $w$ define valuations on $A \otimes B.$  

Now we consider the sum valuation $S \circ v \oplus T \circ w$.  By the first part, without loss of generality, we may take $S, T = 1$.  We must prove that the associated graded algebra is a domain, we claim that it is in fact isomorphic to $gr_v(A) \otimes gr_w(B)$. We compute the quotient $(v \oplus w)_{\leq m}/ (v \oplus w)_{< m}$.

\begin{equation}
(v \oplus w)_{\leq m}/ (v \oplus w)_{< m} = \sum_{i + j \leq m} v_{\leq i} \cap w_{\leq j} \  / \sum_{i + j < m} v_{\leq i} \cap w_{\leq j}\\\\
\end{equation}

\noindent
The space $v_{\leq i}\cap w_{\leq j}$ is equal to $v_{\leq i} \otimes w_{\leq j},$ and $v_{\leq i} \otimes w_{\leq j} \subset v_{\leq i'} \otimes w_{\leq j'}$ if and only if $i \leq i'$ and  $j \leq j'$.  The space $(v \oplus w)_{\leq m}/ (v \oplus w)_{< m}$ is spanned by the images of the spaces $ v_{\leq i} \otimes w_{\leq j}$ with $i + j = m$, and the intersection of two of these spaces $v_{\leq i} \otimes w_{\leq j} \cap v_{\leq i'} \otimes w_{\leq j'}$ is $v_{\leq min\{i, i'\}} \otimes w_{\leq min\{j, j'\}}$, which is contained in $ (v \oplus w)_{< m}$. It follows that the above quotient simplifies to the direct sum of the spaces $v_{\leq i} \otimes w_{\leq j} / v_{< i} \otimes w_{\leq j} + v_{\leq i} \otimes w_{<j} = v_{\leq i}/v_{< i} \otimes w_{\leq j}/ w_{< j},$ this proves the claim. 
\end{proof}

\noindent
Lemma \ref{val2} can be applied to two valuations $v, w: C \to \R \cup \{-\infty\}$ induced from an inclusion of algebras $C \subset A \otimes B$, this is how we use it in Section \ref{charactervaluations}.

We will also need the following lemma, which relates the associated graded algebra of a filtration by representations of a reductive group $G$ to the associated graded algebra of the invariant ring. 

\begin{lemma}\label{val1}
Let $R$ be a $\C$ algebra with a rational action by a reductive group $G$, and let $v$ be a $G-$stable valuation, then the following holds. 

\begin{equation}
gr_v(R^G) = (gr_v(R))^G\\
\end{equation}

\end{lemma}

\begin{proof}
This is a consequence of the fact that passing to invariants by a reductive group is an exact functor.  This fact is applied to the short exact sequence $0 \to v_{< m} \to v_{\leq m} \to v_{\leq m}/v_{< m} \to 0.$
\end{proof}

\subsection{Analytification}

The set of all rank $1$ valuations on a domain $A$ which lift the trivial valuation on $\C$ can be given the structure of a topological space $X^{an},$ called the analytification of the associated affine variety $X = Spec(A),$ see \cite{Ber}, \cite{P}. The analytification is given coarsest topology which makes the evaluation maps $ev_f: X^{an} \to \R$, $ev_f(v) = v(f) \in \R$, $f \in A$ continuous with respect to the usual topology on $\R$. In particular, the topology on $X^{an}$ is generated by the pullbacks of open subsets of $\R$ under the maps $ev_f$; this makes $X^{an}$ into a path connected, locally compact Hausdorff space.  It follows that for a topological space $Y$, a set map $\Phi: Y \to X^{an}$ is continuous if and only if for each $f  \in A,$ the map $[f] \circ \Phi$  defines a continuous map from $Y$ to $\R.$

\section{Degeneration, Compactification and Hamiltonian systems in $SL_2(\C)$}\label{sl2}

In this section we use a valuation to construct a flat degeneration of $SL_2(\C)$ to the singular matrices $SL_2(\C)^c \subset M_2(\C).$  We place symplectic structures on both $SL_2(\C)$ and $SL_2(\C)^c$ and recall a surjective, continuous, map $\Xi: SL_2(\C) \to SL_2(\C)^c$ of Hamiltonian $SU(2) \times SU(2)$ spaces which is a symplectomorphism on a dense open subset of $SL_2(\C)$.     

\subsection{The coordinate ring of $SL_2(\C)$}

The algebraic aspects of $\C[\mathcal{X}(F_g, SL_2(\C))]$ covered in this paper are derived from the interplay between two different ways to view the coordinate ring $\C[SL_2(\C)]$.  As an algebraic variety, the group $SL_2(\C)$ is most familiar as the locus of the equation $AD- BC = 1$ in the space of $2 \times 2$ matrices $M_{2\times 2}(\C)$: 

\begin{equation} SL_2(\C) = \{A, B, C, D | \ det\left[ \begin{array}{cc}
A & B\\
C & D\\
\end{array} \right] = 1 \} \subset M_{2\times 2}(\C).\\
\end{equation}

The Peter-Weyl theorem grants another description of this ring as the direct sum of the endomorphism spaces of the irreducible representations of $SL_2(\C)$.   Recall that there is one such irreducible $V(i)$ for each non-negative integer $i \in \Z_{\geq 0}$, with  $V(i) \cong Sym^i(\C^2)$.  For any $f \in End(V(i))$, there is a regular function on $SL_2(\C)$ defined by $g \to Tr(g^{-1}\circ f)$.  This produces an inclusion of $SL_2(\C) \times SL_2(\C)$ representations: $End(V(i)) \subset \C[SL_2(\C)]$, and defines the $SL_2(\C)\times SL_2(\C)$ isotypical decomposition of $\C[SL_2(\C)]$:

\begin{equation}
\C[SL_2(\C)] = \bigoplus_{i \in \Z_{\geq 0}} End(V(i)).\\ 
\end{equation}

\noindent
An element $(g, h) \in SL_2(\C) \times SL_2(\C)$ acts on $f: V(a) \to V(a)$ by pre and post composition, respectively:  

\begin{equation}
(g, h) \circ [V \to f(V)] = [V \to g\circ f(h^{-1}V)].\\
\end{equation}

\noindent
These descriptions are connected by identifying the generators $A, B, C, D$ above with matrix elements in $End(V(1)) \cong M_{2\times 2}(\C).$  Let $X_{ij} \in End(\C^2)$ be the matrix which has $ij$ entry equal to $1$ and all other entries $0$. The generators of $\C[SL_2(\C)]$ are computed on an element $M \in SL_2(\C)$ as follows:

\begin{equation}
A(M) = Tr(M^{-1}X_{11}) \ \ \ \ B(M) = Tr(M^{-1}X_{01})\\ 
\end{equation}

$$C(M) = Tr(M^{-1}X_{10}) \ \ \ \ D(M) = Tr(M^{-1}X_{00}).$$\\

Let $U_-, U_+ \subset SL_2(\C)$ be the groups of upper, respectively lower triangular matrices in $SL_2(\C).$ It will be necessary to use the coordinate ring of the $GIT$ quotients $SL_2(\C)\q U_-, U_+ \ql SL_2(\C)$, both of which are identified with $\C^2.$  The algebras $\C[SL_2(\C)]^{U_-}, \C[SL_2(\C)]^{U_+}$ are the subspaces of right highest, respectively left lowest weight vectors in the representation $\bigoplus_{i \in \Z_{\geq 0}} End(V(i))$: 

\begin{equation}
\C[SL_2(\C)]^{U_-} = \C[SL_2(\C)]^{U_+} = \bigoplus_{i \in \Z_{\geq 0}} V(i).\\
\end{equation}
\noindent

With actions taken on the right hand side, the algebra $\C[SL_2(\C)]^{U_-} \subset \C[SL_2(\C)]$ is a polynomial ring, generated by $A, C$. Under the automorphism $(-)^{-1}:SL_2(\C) \to SL_2(\C)$, $\C[U_-\ql SL_2(\C)]$ is likewise identified with a polynomial ring in two variables.

The isotypical decomposition of $\C[SL_2(\C)]$ allows an $SL_2(\C) \times SL_2(\C)$-stable filtration defined by the spaces $E_{\leq k} = \bigoplus_{k - i \in 2\Z_{\geq 0}} End(V(i)) \subset \C[SL_2(\C)]$.   The even numbers $2\Z$ in this expression should be interpretted as the root lattice of $SL_2(\C)$.  Following \cite[Chapter 7]{Gr}, the $E_{\leq k}$ define a filtration of algebras, in particular the image of a product $End(V(i))End(V(j))$ lies in the space $\bigoplus_{i + j - k \in 2\Z_{\geq 0}} End(V(k)),$ and the associated graded algebra of this filtration is the algebra $[\C[SL_2(\C)]^{U_-} \otimes \C[SL_2(\C)]^{U_+}]^{\C^*}$.  Taking invariants by the  $\C^*$ action in this expression picks out the invariant subalgebra $\bigoplus_{i \in \Z_{\geq 0}} V(i) \otimes V(i) \subset \C[SL_2(\C)]^{U_-} \otimes \C[SL_2(\C)]^{U_+}$.  Hidden in this statement is the important fact that for any non-zero $f \in End(V(i)), g \in End(V(j))$, the component $(fg)_{i + j} \in End(V(i+j))$ is always non-zero, this is a consequence of the fact that $[\C[SL_2(\C)]^{U_-} \otimes \C[SL_2(\C)]^{U_+}]^{\C^*}$ is a domain.   

We let $SL_2(\C)^c$ be the GIT quotient $[SL_2(\C)\q U_- \times U_+ \ql SL_2(\C)]\q \C^*,$ following the notation in \cite{HMM2}, \cite{M3}; this space is known as the horospherical contraction of $SL_2(\C)$.  Multiplication in $\C[SL_2(\C)^c]$ is graded by the components $V(i) \otimes V(i)$, and is $SL_2(\C) \times SL_2(\C)$-equivariant. The isotypical components are irreducible representations of $SL_2(\C)\times SL_2(\C)$, it follows that each product map $[V(i) \otimes V(i)] \otimes [V(j) \otimes V(j)] \to [V(i + j) \otimes V(i + j)]$ is surjective, and this algebra is generated by the subspace $V(1) \otimes V(1)$.  The associated graded algebra can be presented by $\C[A, B, C, D]$ modulo the initial ideal of $<AD - BC - 1>$ with respect to the filtration $F.$  As $A, B, C, D$ all have the same filtration level, and $1$ is a scalar, we have: 

\begin{equation}
\C[SL_2(\C)^c] = \C[A, B, C, D]/<AD-BC>.\\
\end{equation}

Notice that the equation $AD - BC = 0$ cuts out the set of singular matrices in $M_{2\times 2}(\C).$  We let $X, Y$ be generators of the polynomial coordinate ring $\C[\C^2].$  The algebra $\C[SL_2(\C)^c]$ is isomorphic to $\C[(\C^2\times \C^2)\q \C^*]$ by the map $A \to X \otimes X, B \to X\otimes Y, C \to Y \otimes X, D \to Y \otimes Y.$ 

Let $\mathcal{P}$ be the pointed, polyhedral cone in $\R^3$ defined by the origin $(0, 0, 0)$ and the rays through the points $\{(0, 0, 1), (1, 0, 1), (0, 1, 1), (1, 1, 1)\}$, and  $P$ be the affine semigroup  $\mathcal{P} \cap \Z^3.$ The map $A \to (0, 0, 1),$ $B \to (0, 1, 1),$ $C \to (1, 0, 1),$ $D \to (1, 1, 1)$ defines an isomorphism of algebras, 

\begin{equation}
\C[SL_2(\C)^c] \cong \C[P].\\
\end{equation}

\noindent
We must also consider the Rees algebra of the filtration $E$: 

\begin{equation}
R = \bigoplus_{k \in \Z_{\geq 0}} E_{\leq k}.\\
\end{equation} 

\noindent
This algebra can be presented as $\C[A, B, C, D, t]/<AD - BC - t>$, 
where $t\circ F_{\leq k} \subset F_{\leq k +2}$ identifies $F_{\leq k}$ with itself as a subspace of $F_{\leq k+2}.$  The following are standard properties of Rees algebras, see \cite[3.1]{HMM2}:

\begin{proposition}\label{sl2val}

\begin{enumerate}
\item $R$ is flat over $\C[t]$ (The scheme $Spec(R)$ is a flat family over the affine line.)\\
\item $\frac{1}{t}R = \C[SL_2(\C)]\otimes \C[t, \frac{1}{t}]$ (A general fiber of $Spec(R)$ is isomorphic to $SL_2(\C)$.)\\
\item $R/t = \C[SL_2(\C)^c]$ (The scheme $Spec(R)$ defines a degeneration of $SL_2(\C)$ to $SL_2(\C)^c$)\\
\end{enumerate}

\end{proposition}

\noindent
The total family $t: Spec(R) \to \C$ is the determinant family of $2\times 2$ matrices: $det: M_{2\times 2}(\C) \to \C$.   The space $M_{2\times 2}(\C)$ is the Vinberg enveloping monoid of $SL_2(\C)$, and $SL_2(\C)^c$ is known as the asymptotic semigroup of this monoid, see \cite{Vi1}, \cite{Vi2}, and \cite{HMM2}.

\subsection{Symplectic structures }\label{sl2symplecticstructures}

The discussion of integrable systems in Section \ref{Hamiltonian} requires us to fix symplectic forms on $SL_2(\C)$ and $SL_2(\C)^c$.   All of the forms we use stem from the canonical $SU(2) \times SU(2)-$invariant symplectic form $\omega$ on the cotangent bundle $T^*(SU(2))$.  Recall that $T^*(SU(2))$ is an $SU(2)\times SU(2)$ Hamiltonian space, with right and left momentum maps computed as follows, here $h \circ v$ denotes the coadjoint action of $SU(2)$ on $su(2)^*$.

\begin{equation}
h \circ_L(k, v) = (hk, v) \ \ \ \ \ \ \ h\circ_R(k, v) = (kh^{-1}, h\circ v)\\
\end{equation}

\begin{equation}
\mu_L(k, v) = k \circ v \ \ \ \ \ \ \mu_R(k, v) = -v\\
\end{equation}

Following \cite[Subsection 2.3]{HMM2} and \cite{MaTh}, we induce symplectic form on $SL_2(\C)$ with a Hamiltonian $SU(2)\times SU(2)$ action from the embedding $SL_2(\C) \to M_{2 \times 2}(\C)$, which carries the K\"ahler form $\omega(X, Y) = -Im(Tr(XY^*))$.  Appendix $A$ of \cite{MaTh} and \cite{Sja} then imply that there is an associated isomorphism of Hamiltonian $SU(2)\times SU(2)$ manifolds $SL_2(\C) \cong T^*(SU(2))$.  Letting $\pi: SL_2(\C) \to SU(2)$ be the projection derived from the Cartan polar decomposition, the isomorphism above is computed as below:

\begin{equation}
SL_2(\C) \cong T^*(SU(2)), \ \ \ \  g \to (\pi(g), \mu(g)),\\
\end{equation}

\noindent
where $\mu: M_{2\times 2}(\C) \to su(2)^*$ is the momentum map for the right action of $SU(2)$.

A symplectic form is placed on $SL_2(\C) \q U_-, U_+ \ql SL_2(\C)$ by viewing these spaces as the right, respectively left $symplectic$ $implosions$ of the cotangent bundle with respect to the standard Weyl chamber of $SU(2)$ and its negative, respectively. We refer the reader to \cite{GJS} and \cite{Kir} for the fundamentals of the symplectic implosion operation with respect to a Hamiltonian action by a compact group, a more detailed description appears below. 

We let $\varrho \in su(2)^*$ be the root vector in $su(2)^*$, with $\R_{\geq 0}\varrho, \R_{\leq 0}\varrho \subset su(2)^*$ the positive
and negative Weyl chambers.   The spaces $SL_2(\C) \q U_+$ and $U_- \ql SL_2(\C)$ are identified with quotients of the subspaces $\mu_R^{-1}(\R_{\leq 0}\varrho), \mu_L^{-1}(\R_{\geq 0}\varrho) \subset T^*(SU(2))$ by the equivalence relation which collapses all $(k, v)$ with $v = 0$ to a single point.   The subspaces $\mu_R^{-1}(\R_{\leq 0}\varrho)$ and $\mu_L^{-1}(\R_{\geq 0}\varrho)$ are both topologically half-closed cylinders over a $3-$sphere (homemorphic to $SU(2)$). This equivalence relation collapses the closed sphere at the boundary, creating a topological copy of $\R^4.$  Implosions of Hamiltonian spaces all come with induced, often singular symplectic structures, however the forms on $SL_2(\C) \q U_-$ and $U_+ \ql SL_2(\C)$ both agree with the standard form on $\C^2$, see \cite{GJS}.  These isomorphisms are $S^1$ equivariant with respect to the standard action of $S^1$ on $\C^2$ and the residual $S^1$ action on $SL_2(\C)\q U$ through $\C^*$. In imploded cotangent coordinates, this action is $t\circ_R [k, v] = [kt^{-1}, v]$ and $t\circ_L[k, v] = [tk, v].$ The momentum maps for these actions are given simply by taking the length of the $v$ coordinate. 

 We place a symplectic form on $SL_2(\C)^c$ by taking the symplectic reduction of $SL_2(\C) \q U_- \times U_+ \ql SL_2(\C)$ at level $0$ by the diagonal $S^1$ action.  The momentum $0$ subspace $\mu^{-1}(0) \subset SL_2(\C) \q U_- \times U_+ \ql SL_2(\C)$ is given by those elements $[k, w][h, v]$ with $w = h \circ v.$  By results of Kirwan and Kempf, Ness, \cite{Kir}, \cite{KN}, this symplectic reduction can be identified with $SL_2(\C)^c.$ Notice that this space carries a residual Hamiltonian $S^1$ action, and that the momentum map for this action takes $[k, w][h, v]$ to the length $|v|$.  By \cite[Subsection 5.3.3]{HMM2} this form coincides with the form on $SL_2(\C)^c$ induced from the K\"ahler structure on $M_{2 \times 2}(\C)$.

\subsection{The map $\Xi$}\label{themapxi}

Using the identification $SL_2(\C) \cong T^*(SU(2))$, we may define a continuous function
on $SL_2(\C)$ by setting $\xi(g)$ equal to $|\mu_R(g)|$, the length with respect to the Killing form on $su(2)^*$. This can be computed from the entries of $g$ itself by as follows:

\begin{equation}
\xi(g) = \sqrt{det( g^*g - \frac{1}{2}Tr(g^*g)I)}.\\
\end{equation}

The map $\xi$ is smooth on $SL_2(\C)_o \subset SL_2(\C)$, the space of matrices $g$ with $\mu_R(g) \neq 0$, where it generates a Hamiltonian $S^1$ action.  However $\xi$ is singular along the subgroup $SU(2) \subset SL_2(\C),$ where the $S^1$ flow is consequently not well-defined.  Next we introduce a map which ``repairs'' the undefined flow by relating $SL_2(\C)$ to $SL_2(\C)^c.$  We recall the contraction map $\Xi: SL_2(\C) \to SL_2(\C)^c$ from \cite{M3} and \cite[Section 4]{HMM2}.   For $g \to (k, v)$, let $h \in SU(2)$ be an element which diagonalizes $v$, that is $h \circ v \in \R_{\geq 0}\varrho \subset su(2)^*$, we define $\Xi: SL_2(\C) \to SL_2(\C)^c$ as follows: 

\begin{equation}
\Xi(k, v) = [kh^{-1}, h\circ v][h, v].\\
\end{equation} 

\noindent
It is straightforward to check that $\Xi(k, v)$ defines a point in $SL_2(\C)^c.$
The following proposition recalls all the properties of $\Xi$ we will need.

\begin{proposition}\label{sl2xi}
The map $\Xi: SL_2(\C) \to SL_2(\C)^c$ is well-defined, surjective, continuous, and defines a symplectomorphism on $SL_2(\C)_o.$  The subgroup $SU(2) = SL_2(\C)\setminus SL_2(\C)_o$ is collapsed to a single point in the image.  Furthermore, the composition of $\Xi$ with the residual $S^1$ momentum map $\mu: SL_2(\C)^c \to \R$ is $\xi$.
\end{proposition}

\subsection{The gradient Hamiltonian flow}

The determinant map $det: M_{2 \times 2}(\C) \to \C$ defines an algebraic family joining $SL_2(\C) = det^{-1}(1)$ to $SL_2(\C)^c = det^{-1}(0)$. As explained in \cite{HK} and \cite[Section 5]{HMM2}, the K\"ahler structure on this family can be used to define a gradient Hamiltonian vector field:

\begin{equation}
V_{det} = \frac{\nabla(Re(det))}{|\nabla(Re(det))|}.\\
\end{equation}

\noindent
The work of Harada and Kaveh (making use of a technique of Ruan \cite{R}) implies the existence of a surjective, continuous map $\Xi_{det}: SL_2(\C) \to SL_2(\C)^c$, which is a symplectomorphism on $SL_2(\C)_o$.  Corollary 5.12 of \cite{HMM2} implies that $\Xi_{det}$ coincides with the map $\Xi$ above.

\subsection{The compactification $SL_2(\C) \subset X$}\label{wonderful}

There is another, perhaps more natural filtration on $\C[SL_2(\C)]$ defined by the spaces $F_{\leq k} = \bigoplus_{i \leq k} End(V(i))$.  This is also a filtration of algebras, and the spectrum of the Rees algebra $\bar{R} = \bigoplus_{k \in \Z_{\geq 0}} F_{\leq k}$ is a double cover of $M_{2 \times 2}(\C) = Spec(R)$. The associated graded algebra of this filtration is also isomorphic to $\C[P]$.

The filtration $F$ naturally defines a valuation $v: \C[SL_2(\C)] \to \Z\cup \{-\infty\}$; for $f  = \sum f_i \in \bigoplus_{i \in \Z_{ \geq 0}} End(V(i)),$ one has $v(f) = max\{i | f_i \neq 0\}$.  This function is the divisorial valuation associated to a compactification $X$ of $SL_2(\C)$:\\

\begin{enumerate}
\item $X = Proj(\C[a, b, c, d, t]/<ad - bc - t^2>) \subset \mathbb{P}^4$\\
\item $D = Proj(\C[a, b, c, d]/<ad- bc>) \subset X \subset \mathbb{P}^4$\\
\end{enumerate}

\noindent
The divisor $D$ is cut out by the principal ideal $<t> =  \bigoplus_{a < k} End(V(a))t^k$, and is isomorphic to the projective toric variety $Proj(\C[P]) = \mathbb{P}^1\times \mathbb{P}^1$.  The pullback of $O(1)$ (on $\mathbb{P}^4$) to $D \cong \mathbb{P}^1 \times \mathbb{P}^1$ is $\mathcal{O}(1) \boxtimes \mathcal{O}(1).$ 

It follows from the equation $\frac{1}{t}\C[a, b, c, d, t]/<ad-cb-t^2> = \C[SL_2(\C)]\otimes \C[t, \frac{1}{t}]$ that the complement of $D$ is the scheme $Proj(\frac{1}{t}\C[a, b, c, d, t]/<ad-bc-t>) = Spec(\C[a, b, c, d]/<ad-bc -1> = SL_2(\C).$  We induce a valuation $\bar{v}$ on $\bar{R}$ from the valuation $v \oplus 0$ on $\C[SL_2(\C)] \otimes \C[t].$  The associated graded algebra $T = gr_{\bar{v}}(\bar{R})$ is presented as $\C[a, b, c, d, t]/<ad - bc>$. This defines a flat $SL_2(\C) \times SL_2(\C)-$stable degeneration of $X$ to a projective toric variety $X_0 = Proj(T) =  Proj(\C[P\times \Z_{\geq 0}]).$

\section{Valuations on $\C[\mathcal{X}(F_g, SL_2(\C))]$}\label{charactervaluations}

In this section we construct a cone of valuations $C_{\Gamma, \phi} \subset \mathcal{X}(F_g, SL_2(\C))^{an}$ for each marking $\phi: \Gamma_g \to \Gamma$ in three steps.  First, we define a space $M_{\Gamma}(SL_2(\C))$ for each oriented graph $\Gamma$, such that $M_{\Gamma_g}(SL_2(\C))$ is naturally isomorphic to $\mathcal{X}(F_g, SL_2(\C)).$  Then we construct a simplicial cone of valuations $C_{\Gamma}\subset M_{\Gamma}(SL_2(\C))^{an}$ using the valuation $v: \C[SL_2(\C)] \to \Z \cup \{-\infty\}$ constructed in Section \ref{sl2}. We conclude by showing that each marking $\phi: \Gamma_g \to \Gamma$ gives an isomorphism $\phi^*: M_{\Gamma}(SL_2(\C)) \to M_{\Gamma_g}(SL_2(\C))$, this induces a cone of valuations $C_{\Gamma, \phi} = \phi^*(C_{\Gamma}) \subset \mathcal{X}(F_g, SL_2(\C))^{an}$.

\subsection{The space $M_{\Gamma}(X)$ and the cone $C_{\Gamma}$}

Let $X$ be an affine $SL_2(\C) \times SL_2(\C)$ space, and let $\Gamma$ be any finite oriented graph with all vertices of valence at least three.   We define an action of $SL_2(\C)^{V(\Gamma)}$ on $X^{E(\Gamma)}$, by having an element $(\ldots g_w \ldots)$ $\in SL_2(\C)^{V(\Gamma)}$ act on $(\ldots x_e \ldots) \in X^{E(\Gamma)}$ by sending it to $(\ldots g_vx_eg_u^{-1}\ldots )$, where $\delta(e) = (v, u).$ We let $M_{\Gamma}(X)$ be the corresponding $GIT$ quotient: 

\begin{equation}
M_{\Gamma}(X) = SL_2(\C)^{V(\Gamma)} \ql X^{E(\Gamma)}.\\
\end{equation}

\noindent
Three cases are of this construction are of special interest: $X = SL_2(\C), SL_2(\C)^c,$ and $M_{2 \times 2}(\C)$.  When the graph is $\Gamma_g$ (with any orientation), and $X  = SL_2(\C)$, each $(h_1, \ldots, h_g) \in SL_2(\C)^{E(\Gamma_g)}$ is conjugated by $g_v \in SL_2(\C)^{V(\Gamma_g)} = SL_2(\C)$, so this construction reduces to the $GIT$ definition of the character variety of $F_g$:
 
\begin{equation}
M_{\Gamma_g}(SL_2(\C)) = \mathcal{X}(F_g, SL_2(\C)).\\
\end{equation}

The coordinate ring of each $SL_2(\C)$ component of the product space $SL_2(\C)^{E(\Gamma)}$ comes with a copy of the valuation $v: \C[SL_2(\C)] \to \Z \cup \{-\infty\}$ constructed in Section \ref{sl2}.  We let $v_e$ be the copy of $v$ associated to the the $SL_2(\C)$ component assigned to the edge $e \in E(\Gamma)$. 

By Lemma \ref{val1}, the $\R_{\geq 0}$ combinations of the $v_e$ form a cone of valuations $C_{\Gamma}$ on the coordinate ring $\C[SL_2(\C)^{E(\Gamma)}],$ and therefore $\C[M_{\Gamma}(SL_2(\C))]$.  Let $v_{\ell} = \sum_{e \in E(\Gamma)} \ell(e)v_e$ be the valuation corresponding to an assignment $\ell: E(\Gamma) \to \R_{\geq 0}.$  The following proposition allows us to compute the associated graded algebra of an interior valuation from this cone. 

\begin{proposition}\label{agraded}
For an interior point $\ell \in C_{\Gamma}$,  the associated graded algebra of $\C[M_{\Gamma}(SL_2(\C))]$ with respect to $v_{\ell}$ is isomorphic to $\C[M_{\Gamma}(SL_2(\C)^c)]$. 
\end{proposition}

\begin{proof}
This follows from Lemmas  \ref{val1} and \ref{val2}, and Proposition \ref{sl2val}. 
\end{proof}

\subsection{The graph functor}

We show that the construction $M_{\Gamma}(SL_2(\C))$ is functorial with respect to finite cellular graph maps, in particular any marking $\phi: \Gamma_g \to \Gamma,$ induces an isomorphism $\phi^*: M_{\Gamma}(SL_2(\C)) \cong \mathcal{X}(F_g, SL_2(\C)) = M_{\Gamma_g}(SL_2(\C)).$  We show that if two finite cellular homotopy isomorphisms $\phi_1, \phi_2: \Gamma_1 \to \Gamma_2$ induce the same map on homotopy, $\pi_1(\phi_1) = \pi_1(\phi_2)$, then the maps on varieties also agree, $\phi_1^* = \phi_2^*.$  This means that each equivalence class of markings $[\phi]: \Gamma_g \to \Gamma$ defines a unique associated isomorphism $\phi^*: M_{\Gamma}(SL_2(\C)) \cong \mathcal{X}(F_g, SL_2(\C))$.

A finite cellular map $\phi: \Gamma_1 \to \Gamma_2$ takes an edge $e \in E(\Gamma_1)$ to a finite path of edges $\phi(e) = f(1)\cdots f(k)$ in $\Gamma_2$. These paths determine a regular map of affine varieties $\phi^s: SL_2(\C)^{E(\Gamma_2)} \to SL_2(\C)^{E(\Gamma_1)}$ by sending $(\ldots h_f \ldots) \in  SL_2(\C)^{E(\Gamma_2)}$ to $(\ldots, \prod h_{f(i)}, \ldots) \in  SL_2(\C)^{E(\Gamma_1)}$.  In particular, if $\phi$ collapses the edge $e$, then the associated component is assigned the identity $Id$.  Notice that the value of the $e$-th component of $\phi^s(\ldots, h_f, \ldots)$ only depends on the reduction of the path $\phi(e)$, as the $SL_2(\C)$ elements in any backtracks cancel each other.  In particular, if $e$ is a closed loop in $\Gamma_1$, then the $e$-th component of $\phi^s$ depends only on the minimal length loop in the free homotopy class of $\phi(e)$ (see Proposition \ref{backtrack}).

Let $\phi^a: SL_2(\C)^{V(\Gamma_2)} \to SL_2(\C)^{V(\Gamma_1)}$ send $(\ldots, g_w, \ldots)$ to the tuple defined by the property that $g_u = g_w$ for all $u \in \phi^{-1}(w)$.  
It is clear that $\phi^a$ is a map of reductive groups.

\begin{proposition}\label{graphfunctor}
The map $\phi^s$ descends to define a map of affine varieties on the GIT quotients $\phi^*: M_{\Gamma_2}(SL_2(\C)) \to M_{\Gamma_1}(SL_2(\C))$.  Furthermore, the construction $M_{\Gamma}(SL_2(\C))$ and the maps $\phi^*$ define a functor from the category of finite graphs with finite cellular maps to the category of complex affine varieties. 
\end{proposition}

\begin{proof}
 If we act on $(\ldots h_f \ldots) \in SL_2(\C)^{E(\Gamma_2)}$ with $(\ldots g_w \ldots) \in SL_2(\C)^{V(\Gamma_2)}$ and pass the result through $\phi^s,$ by the definition of the action the resulting element is $g_v(\prod h_{f(i)})g_u^{-1}$, where $\delta(e) = (v, u).$  This calculation implies that $\phi^s$ intertwines the actions of $SL_2(\C)^{V(\Gamma_2)}$ and $SL_2(\C)^{V(\Gamma_1)}$ as-related through $\phi^a$. This implies that $\phi^s$ descends to a map $\phi^*:  M_{\Gamma_2}(SL_2(\C)) \to M_{\Gamma_1}(SL_2(\C))$.  For the second part of the proposition, note that by definition, $(\phi_2 \circ \phi_1)^s = \phi_1^s \circ \phi_2^s$ and $(\phi_2 \circ \phi_1)^a = \phi_1^a \circ \phi_2^a$, as a consequence $(\phi_2 \circ \phi_1)^* = \phi_1^* \circ \phi_2^*.$   
\end{proof}

\begin{example}
Let $\phi: \Gamma \to \Gamma$ be an isomorphism of graphs, which reverses the orientation of one edge $e\in E(\Gamma)$ with $\delta(e) = (u, v)$. This induces an automorphism of $M_{\Gamma}(SL_2(\C))$ by $(-)^{-1}: SL_2(\C)^{e} \to SL_2(\C)^{e} \subset SL_2(\C)^{E(\Gamma)}.$  This takes the edge element $g_uh_eg_v^{-1}$ to $g_vh_e^{-1}g_u^{-1},$ and therefore intertwines the $SL_2(\C)^{V(\Gamma)}$ action, giving an isomorphism of varieties.  
\end{example}

\begin{example}\label{freegroupaction}
Let $\theta: \Gamma_g \to \Gamma_g$ be induced from an outer automorphism of $F_g$, namely an assignment $e_i \to W_i$ of $g$
generating words to the edges of $\Gamma_g$.  Then $\theta^s: SL_2(\C)^{E(\Gamma_g)} \to SL_2(\C)^{E(\Gamma_g)}$
is the regular map defined by sending $\vec{h} = (h_1, \ldots, h_g) \to (W_1(\vec{h}), \ldots, W_g(\vec{h})).$ The inverse of this map is likewise defined by the inverse of $\theta$.  These automorphisms define the action of $Out(F_g)$ on $\mathcal{X}(F_g, SL_2(\C)) = M_{\Gamma_g}(SL_2(\C)).$
\end{example}

\begin{example}\label{collapsingmap}
Let $\phi_{\tree}: \Gamma \to \Gamma_g$ and $\psi_{V, \tree}: \Gamma_g \to \Gamma$ be the pair of maps induced by a choice of spanning tree $\tree \subset \Gamma.$   

The map $\phi_{\tree}^s$ sends $(h_1, \ldots, h_g) \in SL_2(\C)^{E(\Gamma_g)}$ to the element in $SL_2(\C)^{E(\Gamma)}$ obtained by assigning $Id$ to all edges in $\tree$, and $h_i$ to the edge in $E(\Gamma) \setminus E(\tree)$ corresponding to the $i-$th edge of $\Gamma_g$.   As $\phi_{\tree}$ maps each vertex in $\Gamma$ down to the unique vertex of $\Gamma_g$, $\phi_{\tree}^a: SL_2(\C) \to SL_2(\C)^{V(\Gamma)}$ is the diagonal map of reductive groups.   

The map $\psi_{V, \tree}^s$ sends an element $(\ldots h_e \ldots) \in SL_2(\C)^{E(\Gamma)}$ to an assignment of words $w_i(\ldots h_e \ldots)$ on the edges of $\Gamma_g$.  Let $e_i \in E(\Gamma)$ be the preimage of the $i-$th edge of $\Gamma_g$ under $\phi_{\tree}$, with $\delta(e_i) = (u_i, v_i).$  The word $w_i(\ldots h_e \ldots)$ is computed by concatenating the unique path $V \to u_i$ in $\tree$ with  $e_i$ and the unique path $v_i \to V$ in $\tree$, and recording $h_e$ or $h_e^{-1}$ for each traversed edge, where the sign depends on the orientation of $e$ with respect to the direction of the path. 
\end{example}

\begin{lemma}\label{distiso}
For any distinguished map $\psi_{V, \tree}: \Gamma_g \to \Gamma$, the induced map $\psi_{V, \tree}^*$ is an isomorphism of affine schemes.  
\end{lemma}

\begin{proof}
We must check that $\psi_{V, \tree}^*$ and $\phi_{\tree}^*$ discussed in Example \ref{collapsingmap} are inverse to each other.  The discussion in Example \ref{collapsingmap} immediately implies that $\psi_{V, \tree}^* \circ \phi_{\tree}^*$ is the identity on $M_{\Gamma_g}(SL_2(\C)).$  

Consider any edge $e \in \tree \subset \Gamma$ with $\delta(e) = (u, v)$, $v \neq V.$  Using the action of $SL_2(\C)^{\{v\}} \subset SL_2(\C)^{V(\Gamma)}$ we can move $(\ldots h_e \ldots)$ to $(\ldots, 1, \ldots)$, by introducing $h_e$ or $h_e^{-1}$ on the other incident edges of $v$ according to their orientations.  The $W_i$ do not change under this operation.  The same construction can be applied to $u \neq V$ as above.  It follows that we may move any point in $SL_2(\C)^{E(\Gamma)}$ into the image of $\phi^* \circ \psi^*$ using the action of $SL_2(\C)^{V(\Gamma)}$. 
\end{proof}

Notice that Lemma \ref{distiso} implies that $\psi^*_{V, \tree}$ does not depend on the choice $V \in V(\Gamma)$.  Indeed, the image $\psi_{V, \tree}^*(\ldots h_e \ldots) = (\ldots W_i(\ldots h_e \ldots) \ldots)$ differs under a change of basepoint by including or taking away backtracks in the words $W_i$, a modification which does not change the value of $W_i(\ldots, h_e, \ldots ) \in SL_2(\C)$.  Next we prove a generalization of this fact.

\begin{proposition}\label{hominv}
Let $\phi_1, \phi_2: \Gamma_g \to \Gamma$ be markings which satisfy  $\pi_1(\phi_1) = \pi_1(\phi_2)$, then $\phi_1^* = \phi_2^*.$ 
\end{proposition}

\begin{proof}
For any edge $e \in E(\Gamma_g)$, the loops $\phi_1(e), \phi_2(e) \subset \Gamma$ are freely homotopic.  It follows that both differ from a common minimal length cellular class $\gamma$ by backtracks.  Therefore, we must have $W^1_i(\ldots h_e \ldots) = W^2_i(\ldots h_e \ldots)$ for any $(\ldots h_e \ldots) \in M_{\Gamma}(SL_2(\C)).$
\end{proof}

\begin{corollary}\label{hominv2}
Let $\phi_1, \phi_2: \Gamma_1 \to \Gamma_2$ satisfy $\pi_1(\phi_1) = \pi_1(\phi_2)$, then $\phi_1^* = \phi_2^*$. 
\end{corollary}

\begin{proof}
We can choose a spanning tree in $\Gamma_1$, and a corresponding distinguished map $\theta: \Gamma_g \to \Gamma_1$. This defines two cellular maps $\theta_i = \phi_i \circ \theta: \Gamma_g \to \Gamma_2$, which are the same map under $\pi_1.$  It follows from the proof of Proposition \ref{hominv} that $\theta_1^*, \theta_2^*$ and therefore  $\phi_1^*, \phi_2^*$ define the same maps on varieties, see Figure \ref{homequiv}.
\end{proof}

From Proposition \ref{hominv} and Lemma \ref{distiso}, it follows that any equivalence class of markings $[\phi]: \Gamma_g \to \Gamma$ corresponds to a well-defined isomorphism of varieties $\phi^*: M_{\Gamma}(SL_2(\C)) \cong \mathcal{X}(F_g, SL_2(\C)),$ and an isomorphism of coordinate rings: 

\begin{equation}
\hat{\phi}: \C[\mathcal{X}(F_g, SL_2(\C))] \to \C[M_{\Gamma}(SL_2(\C))].\\
\end{equation}

\noindent
As a consequence, there is a cone of valuations $C_{\Gamma, \phi} \subset \mathcal{X}(F_g, SL_2(\C))^{an}$ obtained by pulling back $C_{\Gamma}$ along $\hat{\phi}.$

\begin{figure}[htbp]

$$
\begin{xy}
(0, -3)*{\mathcal{X}(F_g, SL_2(\C))} = "A";
(-8.5, 0)*{} = "A0";	
(3.5, 0)*{} = "A1";	
(-3.5, 0)*{} = "A2";
(28, 30)*{M_{\Gamma_1}(SL_2(\C))} = "B";
(20,25)*{} = "B1";
(15, 28)*{} = "B2";
(15, 33)*{} = "B3";
(-28, 30)*{M_{\Gamma_1}(SL_2(\C))} = "C";
(-25,25)*{} = "C0";
(-20,25)*{} = "C1";
(-15, 28)*{} = "C2"; 
(-15, 33)*{} ="C3";
(0, 30)*{\phi_1^*};
(0, 35)*{\phi_2^*};
(18, 14)*{\theta^*};
(-15, 14)*{\theta_1^*};
(-20, 14)*{\theta_2^*};
(20, 14)*{};
{\ar@{>}"B1"; "A1"};
{\ar@{>}"C0"; "A0"};
{\ar@{>}"C1"; "A2"};
{\ar@{>}"C2"; "B2"};
{\ar@{>}"C3"; "B3"};
\end{xy}
$$\\
\caption{}
\label{homequiv}
\end{figure}

\section{The sets $\mathcal{R}(\Gamma, \phi), \mathcal{S}(\Gamma, \phi)$ and length functions}\label{spanningsets}

We discuss two spanning sets $\mathcal{R}(\Gamma, \phi), \mathcal{S}(\Gamma, \phi) \subset \C[\mathcal{X}(F_g, SL_2(\C))]$ associated to a marking $\phi: \Gamma_g \to \Gamma$.  Elements of $S(\Gamma, \phi)$ correspond to arrangements of closed loops in $\Gamma$, and the elements of $R(\Gamma, \phi)$ are in bijection with the so-called spin diagrams with topology $\Gamma$.   We relate these sets to each other, showing that any element of $\mathcal{S}(\Gamma, \phi)$ has a lower-triangular expansion into spin diagrams in $\mathcal{R}(\Gamma, \phi)$ with respect to a natural partial ordering.  We then show how to evaluate elements from both of these sets in the valuations $v_{\Gamma, \ell, \phi} \in C_{\Gamma, \phi}.$  We finish the section by showing that the set $\mathcal{S}(\Gamma, \phi)$ does not depend on the marking $\phi$ or the graph $\Gamma$, and we use this to give an embedding of $\hat{O}(g)$ into $\mathcal{X}(F_g, SL_2(\C))^{an}.$

\subsection{Spin diagrams}\label{subspindiagrams}

We fix a marking of a trivalent graph $\phi: \Gamma_g \to \Gamma$, and use the induced isomorphism $\phi^*: M_{\Gamma}(SL_2(\C)) \to \mathcal{X}(F_g, SL_2(\C))$ to define a basis in $\C[\mathcal{X}(F_g, SL_2(\C))]$ with invariant theory.   Observe that the coordinate ring of the space $SL_2(\C)^{E(\Gamma)}$ has the following isotypical decomposition under the action of $SL_2(\C)^{E(\Gamma)} \times SL_2(\C)^{E(\Gamma)}$. 

\begin{equation}
\C[SL_2(\C)^{E(\Gamma)}] = \bigoplus_{a: E(\Gamma) \to \Z_{\geq 0}} \bigotimes_{e \in E(\Gamma)} V(a(e))\otimes V(a(e))\\
\end{equation}

\noindent
Passing to $SL_2(\C)^{V(\Gamma)}$ invariants gives a direct sum decomposition of the coordinate ring of $M_{\Gamma}(SL_2(\C))$. 

\begin{equation}
\C[M_{\Gamma}(SL_2(\C))] =  \bigoplus_{a: E(\Gamma) \to \Z_{\geq 0}} [\bigotimes_{e \in E(\Gamma)} V(a(e))\otimes V(a(e))]^{SL_2(\C)^{V(\Gamma)}}\\
\end{equation}

Recall that there is a copy of $SL_2(\C)$ for each vertex $v \in V(\Gamma)$ acting on the tensor product $V(a(e)) \otimes V(a(f)) \otimes V(a(g))$, where $v \in \delta(e), \delta(f), \delta(g).$  The Clebsch-Gordon rule implies that the invariant subspace of such a tensor product is at most one dimensional, and this space is nontrivial if and only if the numbers $a(e), a(f), a(g)$ satisfy two conditions.

\begin{enumerate}
\item $a(e) + a(f) + a(g) \in 2\Z$\\
\item $a(e), a(f), a(g)$ are the sides of a triangle: $|a(e) - a(g)| \leq a(f) \leq a(e) + a(g)$.\\
\end{enumerate}

\begin{definition}\label{spincone}
We let $\mathcal{P}_{\Gamma}$ be the polyhedral cone of $a: E(\Gamma) \to \R_{\geq 0}$ which satisfy condition $2$ above, and we let $L_{\Gamma} \subset \Z^{E(\Gamma)}$ be the lattice defined by condition $1$ above.  Finally,  $P_{\Gamma} = \mathcal{P}_{\Gamma} \cap L_{\Gamma}$ is defined to be the associated affine semigroup.  
\end{definition}

For $a: E(\Gamma) \to \Z_{\geq 0}$ the invariant spaces $[V(a(e)) \otimes V(a(f)) \otimes V(a(g))]^{SL_2(\C)}$ are multiplicity-free, it follows that each space $ [\bigotimes_{e \in E(\Gamma)} V(a(e))\otimes V(a(e))]^{SL_2(\C)^{V(\Gamma)}}$ is multiplicity-free as well, with dimension $1$ occurring precisely when $a \in P_{\Gamma}$.  We fix a non-zero element $\Phi_a \in  [\bigotimes_{e \in E(\Gamma)} V(a(e))\otimes V(a(e))]^{SL_2(\C)^{V(\Gamma)}}$.  This choice defines a direct sum decomposition: 

\begin{equation}
\C[M_{\Gamma}(SL_2(\C))] = \bigoplus_{a \in P_{\Gamma}} \C\Phi_a\\
\end{equation}

\noindent
Let $R(\Gamma) \subset \C[M_{\Gamma}(SL_2(\C))]$ be the set composed of the $\Phi_a$; these functions are called spin diagrams.  For any valuation $v_{\ell} \in C_{\Gamma}$ the following formula holds by definition (see Figure \ref{spinval}): 

\begin{equation}\label{innerproducteval}
v_{\ell}(\Phi_a) = \sum_{e \in E(\Gamma)} \ell(e)a(e).\\
\end{equation}

\begin{figure}[htbp]
\centering
\includegraphics[scale = 0.4]{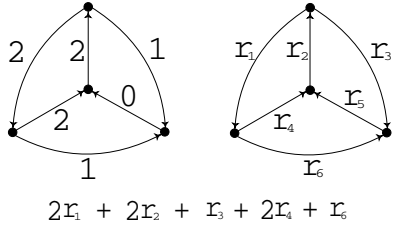}
\caption{Evaluating a spin diagram element on the left in a valuation on the right.}
\label{spinval}
\end{figure}

\begin{proposition}\label{spinadapt}
For any $\ell \in C_{\Gamma}$, the space $v_{\ell}^{\leq r}$ is equal to the direct sum $\bigoplus_{a | v_{\ell}(\Phi_a) \leq r} \C\Phi_a$.  Furthermore, for any $f \in \C[M_{\Gamma}(SL_2(\C))]$ with $f = \sum C_a\Phi_a$ we have:

\begin{equation}
v_{\ell}(f) = MAX\{v_{\ell}(\Phi_a) | C_a \neq 0\}.\\
\end{equation}

\end{proposition}

\begin{proof}
This is a consequence of the definition of the space $\C\Phi_a$ and Lemmas \ref{val1} and \ref{val2}.
\end{proof}

We place a partial ordering on the lattice points in $P_{\Gamma},$ where $a \preceq a'$ if $a' - a \in C_{\Gamma}.$

\begin{proposition}
A product $\Phi_a \Phi_{a'}$ is a linear combination of $\Phi_{t}$ with $t \preceq a + a'$, furthermore, the term $\Phi_{a + a'}$ always appears with a non-zero coefficient.
\end{proposition}

\begin{proof}
The first statement is a consequence of the corresponding fact for $\C[SL_2(\C)]$.  For the second statement, we choose a valuation $v_{\ell}$ with $\ell$ in the interior of $C_{\Gamma}$, and consider $v_{\ell}(\Phi_a\Phi_{a'}) = v_{\ell}(\Phi_a) + v_{\ell}(\Phi_{a'})$.  If $\Phi_{a + a'}$ did not appear in $\Phi_a \Phi_{a'}$ then this element would be a linear combination of $\Phi_t$ with $t(e) < a(e) + a'(e)$ for
some $e \in E(\Gamma)$.  For each of these terms we must have $v_{\ell}(\Phi_t) < v_{\ell}(\Phi_a) + v_{\ell}(\Phi_{a'})$ by Formula \ref{innerproducteval}, which contradicts that $v_{\ell}$ is  a  valuation. 
\end{proof}

For any $f \in \C[M_{\Gamma}(SL_2(\C))]$ we may write $f = \sum_{a \in P_{\Gamma}} C_a\Phi_a$.  If there is a $a \in P_{\Gamma}$ with $C_a \neq 0$ and $a' \prec a$ for all $C_{a'} \neq 0$, we say that $\Phi_a$ is the initial term of $f$, $in(f) = C_a\Phi_a$.  This notion can also be defined in the coordinate ring $\C[SL_2(\C)^{E(\Gamma)}]$, where $in(f)$ is taken to be the contribution from the isotypical space $\otimes_{e \in E(\Gamma)} V(a(e)) \otimes V(a(e))$ with $a$ maximal under $\prec$.  We note that the $SL_2(\C)^{V(\Gamma)}$-stability of this decomposition implies that the initial term of an invariant $f \in \C[M_{\Gamma}(SL_2(\C))] \subset \C[SL_2(\C)^{E(\Gamma)}]$ can be computed in either ring, with the same result. 

The coordinate ring $\C[M_{\Gamma}(SL_2(\C)^c)]$ has a decomposition into invariant spaces identical to the decomposition of $\C[M_{\Gamma}(SL_2(\C))]$ discussed above.  Each of these spaces is of course $1-$dimensional, and the multiplication operation on $\C[SL_2(\C)^c]$ is graded by the highest weights in the isotypical decomposition. The following is an immediate consequence of these observations. 

\begin{proposition}\label{contractiontoricvariety}
The space $M_{\Gamma}(SL_2(\C)^c)$ is the affine toric variety $Spec( \C[P_{\Gamma}])$. 
\end{proposition}

\subsection{An alternative construction of $M_{\Gamma}(SL_2(\C))$.} 

We give a different construction of the space $M_{\Gamma}(SL_2(\C))$ which is useful for describing the second spanning set $S(\Gamma) \subset \C[M_{\Gamma}(SL_2(\C))]$.  Consider the following left diagonal $GIT$ quotient: 

\begin{equation}
M_{k}(SL_2(\C)) = SL_2(\C) \ql SL_2(\C)^k.\\
\end{equation}

\noindent
This space retains a right action by $SL_2(\C)^k$.

For an oriented graph $\Gamma$, we form the product variety $\prod_{v \in V(\Gamma)} M_{\epsilon(v)}(SL_2(\C))$.  This product is taken over vertices so each edge $e \in E(\Gamma)$ is represented twice, and each corresponding copy of $SL_2(\C)$ has a right action by $SL_2(\C)$ as above.   Consequently, this product space comes with a residual action by $SL_2(\C)^{E(\Gamma)} \times SL_2(\C)^{E(\Gamma)}$. A right hand side quotient by a diagonally embedded $SL_2(\C)^{E(\Gamma)} \subset SL_2(\C)^{2E(\Gamma)}$ gives the following identity: 

\begin{equation}
[\prod_{v \in V(\Gamma)} M_{\epsilon(v)}(SL_2(\C))] \q SL_2(\C)^{E(\Gamma)} \cong\\ \end{equation}

$$SL_2(\C)^{V(\Gamma)} \ql [\prod_{e \in E(\Gamma)} ([SL_2(\C)^2]\q SL_2(\C))^{E(\Gamma)}] =   M_{\Gamma}(SL_2(\C)).$$

\begin{figure}[htbp]
\centering
\includegraphics[scale = 0.6]{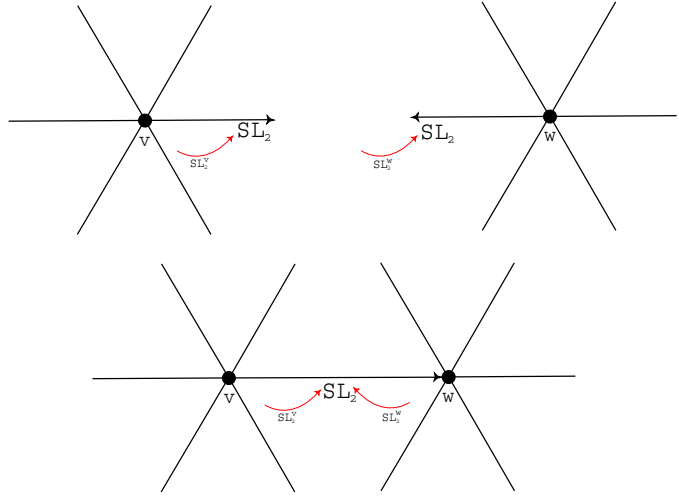}
\caption{The isomorphism takes $SL_2(\C) \ql [SL_2(\C) \times SL_2(\C)]$ with its two right hand actions to $SL_2(\C)$ with its left and right actions.}
\label{Gluing}
\end{figure}

\noindent 
To see this, we view the group $SL_2(\C)$ as an affine $GIT$ quotient $[SL_2(\C) \times SL_2(\C)]\q SL_2(\C),$ where $SL_2(\C)$ acts on the right hand side of both components. The equivariant $SL_2(\C) \times SL_2(\C)$ isomorphism between these spaces is provided by the maps $(g, h) \to h^{-1}g$, $g \to (g, 1)$.

For an oriented tree $\tree$ we define a space $M_{\tree}(SL_2(\C))$ using the same quotient recipe:

\begin{equation}
M_{\tree}(SL_2(\C)) = [\prod_{v \in V(\tree)} M_{\epsilon(v)}(SL_2(\C))]\q SL_2(\C)^{E(\tree)}.\\
\end{equation}

\begin{lemma}\label{distisotree}
For a tree $\tree$ with leaf set $\mathcal{L}(\tree)$ there is an isomorphism:

\begin{equation}
\Phi_{\tree}: M_{|\mathcal{L}(\tree)|}(SL_2(\C)) \to M_{\tree}(SL_2(\C)).\\
\end{equation}

\end{lemma}

\begin{proof}
 By repeatedly applying the isomorphism $[SL_2(\C) \times SL_2(\C)]\q SL_2(\C) \cong SL_2(\C)$, we find that $M_{\tree}(SL_2(\C))$ is isomorphic to the quotient space:

\begin{equation}
SL_2(\C)^{V(\tree)} \ql [SL_2(\C)^{E(\tree) \cup \mathcal{L}(\tree)}].\\
\end{equation} 

For ease of notation we set $n = |\mathcal{L}(\tree)|$, and we let $\Phi_{\tree}: M_{\mathcal{L}(\tree)}(SL_2(\C)) \to M_{\tree}(SL_2(\C))$ be the map that assigns $(g_1, \ldots, g_n)$ to the leaves of $\tree$ and $Id \in SL_2(\C)$ to all edges, similar to the map $\phi_{\Gamma}$ from Section \ref{charactervaluations}. 

We choose a base point $V \in V(\tree)$ and define the map $\Psi_{V, \tree}: M_{\tree}(SL_2(\C)) \to M_{\mathcal{L}(\tree)}(SL_2(\C))$.  This map sends an element $(\ldots, g_e, \ldots)$ is sent to the element $(W_1(\ldots g_e \ldots), \ldots, W_n(\ldots g_e \ldots))$ where $W_i$ is the oriented word obtained by following the unique path from $V$ to the $i-$th leaf of $\tree.$  Now the proof of Lemma \ref{distiso} shows that $\Phi_{\tree}$ and $\Psi_{V, \tree}$ are inverses. 
\end{proof}

For a graph $\Gamma$, and a spanning tree $\tree$, we let $T(\tree, \Gamma)$ be the tree obtained by splitting each edge in $E(\Gamma) \setminus E(\tree),$ see Figure \ref{Spantree}.  By once again making use of the isomorphism $[SL_2(\C) \times SL_2(\C)] \q SL_2(\C) \cong SL_2(\C)$, we observe that there is a natural quotient map associated to gluing the two ends of a split edge back together:

\begin{equation}
\pi_{\tree, \Gamma}: M_{T(\tree, \Gamma)}(SL_2(\C)) \to M_{\Gamma}(SL_2(\C)) = M_{T(\tree, \Gamma)}(SL_2(\C))\q SL_2(\C)^{E(\Gamma) \setminus E(\tree)}.\\
\end{equation}

\noindent
We let $\pi_0$ be this quotient map for  the special case $\Gamma = \Gamma_g$. 

\begin{equation}
\pi_0: M_{2g}(SL_2(\C)) \to M_{\Gamma_g}(SL_2(\C))\\ 
\end{equation}

\begin{figure}[htbp]
\centering
\includegraphics[scale = 0.6]{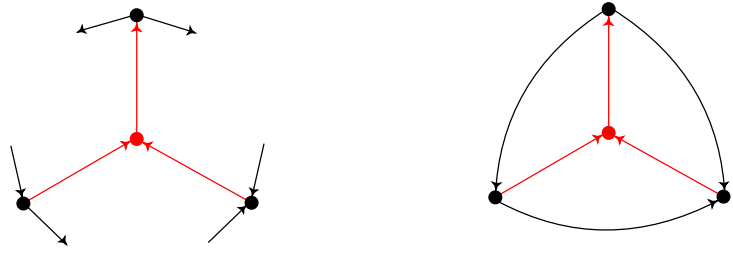}
\caption{A graph $\Gamma$ on the right, with spanning tree $\tree$ in red, and the tree $T(\tree, \Gamma)$ on the left.}
\label{Spantree}
\end{figure}

The following proposition allows us to relate our constructions on the graph $\Gamma$ to those on the tree $T(\tree, \Gamma),$ it is a consequence of the argument used in Lemmas \ref{distiso} and \ref{distisotree}.

\begin{proposition}\label{spanningtreespace}
The maps $\Phi_{T(\tree, \Gamma)}$ and $\Psi_{T(\tree, \Gamma)}$ are isomorphisms.  Furthermore, the following diagram commutes: 

\begin{equation}
\xymatrix{
  M_{T(\tree, \Gamma)}(SL_2(\C)) \ar[r]^{\pi_{\tree, \Gamma}} \ar@<-2pt>[d]_{\Psi_{V, \tree}} & M_{\Gamma}(SL_2(\C)) \ar@<-2pt>[d]_{\psi_{V, \Gamma}} \\
   M_{2g}(SL_2(\C)) \ar@<-2pt>[u]_{\Phi_{\tree}} \ar[r]^{\pi_0} & M_{\Gamma_g}(SL_2(\C)) \ar@<-2pt>[u]_{\phi_{\Gamma}}.
}
\end{equation}

\end{proposition}

Now we give a presentation of the coordinate ring
$\C[M_{n}(SL_2(\C))]$ (this also appears in \cite[Section 8]{M15}).  The space $SL_2(\C)^n$  is an affine subspace of the space of $2 \times 2n$ matrices, cut out by $n$ determinant equations.  Below $C_{i, j}$ is a column of an element of $M_{2, 2n}(\C)$:

\begin{equation}
SL_2(\C)^n = \{[C_{1, 1}, C_{1, 2}, \ldots, C_{n, 1}, C_{n, 2}] \in M_{2, 2n}(\C) | det(C_{i,1}, C_{i,2}) = 1\}. \\
\end{equation}

We can realize $M_{n}(SL_2(\C))$ as a subvariety of $SL_2(\C) \ql M_{2, 2n}(SL_2(\C))$.
Following Weyl's First Fundamental Theorem of Invariant Theory, the latter is cut out of $\C^{\binom{2n}{2}}$ by Pl\"ucker equations.
The invariant ring $\C[M_{2, 2n}(SL_2(\C))]^{SL_2(\C)}$ is generated by the forms $p_{(i,a), (j, b)} = det(C_{i,a}, C_{j,b})$.  The space $M_{n}(SL_2(\C))$ is then the intersection of the hypersurfaces in this quotient determined by the equations $p_{(i, 1), (j, 2)} = 1.$

\begin{proposition}
Order the indices $(i, a)$ lexicographically.  The coordinate ring of $M_{n}(SL_2(\C))$ is generated by the $\binom{2n}{2}$ generators $p_{(i, a)},$ subject to the following equations,  where $(i, a) < (j, b) < (k, c) < (l, d)$.  

\begin{equation}
 p_{(i, a), (j, b)} p_{(k, c), (l, d)} +  p_{(i, a), (l, d)} p_{(j, b), (k, c)} =  p_{(i, a), (k, c)} p_{(j, b), (l, d)}, \ \ \ \ \ p_{(i, 1),(i,2)} = 1\\
\end{equation}

\end{proposition}

The algebra $\C[M_{n}(SL_2(\C))]$ controls the structure of $M_{\Gamma}(SL_2(\C))$ ``at a vertex of valence $n$.'' Each index $i$ is associated to an incident edge, and the ordering of the indices $i$ above is tantamount to a choice of a cyclic ordering of the edges incident on such a vertex.  For a graph $\Gamma$ such a choice is known as a ribbon structure on $\Gamma$.

  We will also make use of the "opposite" Pl\"ucker generators $-p_{(i, a)(j, b)} = p_{(j, b)(i, a)}$ when direction is important.    Of course, $M_{n}(SL_2(\C)) = SL_2(\C)^{n-1}$, so this is perhaps an overly complicated presentation, however the combinatorics of the Pl\"ucker generators and equations play an important role in describing the set $S(\Gamma) \subset \C[M_{\Gamma}(SL_2(\C))]$.  Throughout we will use ``direction" to mean the ordering of indices, as in $p_{(i, a), (j, b)}$ goes from $i$ to $j$, and ``sign" to mean the data $a, b$.

\subsection{The space $M_{3}(SL_2(\C))$}

Let $v_1, v_2, v_3$ be three copies of $v: \C[SL_2(\C)] \to \Z \cup \{-\infty\}$ defined on $\C[M_{3}(SL_2(\C))] \subset \C[SL_2(\C)^3].$
We apply Lemmas \ref{val1} and \ref{val2} to show that the associated graded algebra of an interior element of $C_3 = \R_{\geq 0}\{v_1, v_2,  v_3\}$ is $\C[SL_2(\C)^c \times SL_2(\C)^c \times SL_2(\C)^c]^{SL_2(\C)} = \C[M_{3}(SL_2(\C)^c)],$ and we
show that this is an affine semigroup algebra. 

\begin{definition}
The cone $\mathcal{P}_3$ is defined to be the set of labelings of the following diagram by non-negative real numbers
which satisfy $x_1 + y_1 = a, x_2 + y_2 = b, x_3 + y_3 = c$, such that $a, b, c$ form the sides of a triangle, $|a-c| \leq b \leq a+c.$ 
The lattice $L_3$ is the set of integer labelings of this diagram with  $a + b +c \in 2\Z$.  We let $P_3$ be the affine semigroup $\mathcal{P}_3 \cap L_3$, see Figure \ref{P3el}. 

\begin{figure}[htbp]
\centering
\includegraphics[scale = 0.4]{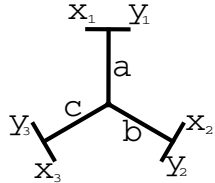}
\caption{Defining diagram of $P_3$.}
\label{P3el}
\end{figure}

\end{definition}

\begin{proposition}
The algebra $\C[M_{3}(SL_2(\C)^c)]$ is isomorphic to the affine semigroup algebra $\C[P_3]$.  
\end{proposition}

\begin{proof}
Following Section \ref{sl2}, $M_{3}(SL_2(\C)^c)$ can be rewritten as the $GIT$ quotient:

\begin{equation}
SL_2(\C) \ql [(SL_2(\C)\q U_+)^3] \times [(U_- \ql SL_2(\C))^3]\q (\C^*)^3.\\
\end{equation}

\noindent
We compute this as a reduction in steps.  The quotient $SL_2(\C) \ql [(SL_2(\C)\q U_+)^3] = SL_2(\C) \ql [(\C^2)^3]$
is isomorphic to the affine space $\bigwedge^2(\C^3)$ by the First Fundamental Theorem of Invariant Theory. It can be viewed
as the variety associated to the affine semigroup algebra $\C[Q_3]$, where $Q_3$ is the set
of triples $a, b, c$ as in the definition of $P_3$ above. Here $a, b, c$ are the highest weights of the components
$(V(a)\otimes V(b) \otimes V(c))^{SL_2(\C)} \subset SL_2(\C) \ql [(\C^2)^3]$.  Recall that the character spaces
of the diagonal action of $\C^*$ on $SL_2(\C)\q U = \C^2$ are precisely the $V(a) \subset \C[\C^2].$  It follows
that $\C[M_{3}(SL_2(\C))]$ is graded by the subspaces $((V(a)\otimes V(b) \otimes V(c))^{SL_2(\C)} \otimes V(a) \otimes V(b) \otimes V(c)$. 
Since $\C[U_- \ql SL_2(\C)] = \C[\C^2]$ is itself toric, we can further grade this algebra by viewing each $V(a)$ as a direct
sum of the monomial spaces $\C X^{x_1}Y^{y_1}$, where $x_1 + y_1 = a.$ A similar decomposition holds for $V(b), V(c).$ 
\end{proof}

The affine semigroup $P_3$ is generated by the $3 \times 4 = 12$ weightings with one of $a, b, c$ equal to $0$
and the other two entries equal to $1.$ We label these generators $X_{(i, s), (j, t)}$, indicating a path from $i$ to $j$
with sign markings on the $s, t$ ends of these paths, respectively. Now we recall the forms $p_{(i, s) (j, t)} \in \C[M_{3}(SL_2(\C))]$.  If we view $M_{3}(SL_2(\C))$ as the space of $2 \times 6$ matrices $[C_{1, 1}C_{1, 2}C_{2,1}, C_{2,2}, C_{3,1}C_{3,2}]$
which satisfy $det(C_{i,1}C_{i,2}) = 1,$ $p_{(i,s),(j,t)}$ is the function $det(C_{i,s}C_{j,t})$.   We let the initial form $in(f) \in \C[M_{3}(SL_2(\C)^c)]$ (when it exists) for $f \in \C[M_{3}(SL_2(\C))]$ be as in Subsection \ref{subspindiagrams}.

\begin{proposition}
The initial forms $in(p_{(i,s),(j,t)})$ give a generating set of $P_3.$
\end{proposition}

\begin{proof}
Recall that the algebra $\C[SL_2(\C)^c]$ is isomorphic to $\C[\C^2 \times \C^2]^{\C^*}$, and the isomorphism does the following to generators:
$A \to X \otimes X$, $B \to X \otimes Y$, $C \to Y \otimes X$, $D \to Y \otimes Y$.  We consider the image of the tensor $p_{(1, 1), (2, 1)} \in \C[M_{3}(SL_2(\C)^c],$  given by the determinant $A_1C_2 - A_2C_1 =$ $(X_1 \otimes X_1)(Y_2 \otimes X_2) -$ $(X_2 \otimes X_2)(Y_1 \otimes X_1)=$ $(X_1Y_2 - X_2Y_1) \otimes X_1X_2.$  We obtain the Pl\"ucker invariant $X_1Y_2 - X_2Y_1 \in \C[\C^2 \times \C^2 \times \C^2]^{SL_2(\C)}$ $=\C[M_{3}(SL_2(\C)\q U_-^3]$ tensored with the leaf data $X_1X_2$, this is precisely the element $X_{(1, 1), (2, 1)} \in P_3$.  This computation works for all generators $X_{(i, a)(j, b)}.$  

\end{proof}

Note that if generators $p_{(i_1, a_1), (j_1, b_1)}, \ldots, p_{(i_n, a_n), (j_n, b_n)}$ are multiplied, 
the $x_i, y_i$ components of the initial term $in(\prod p_{(i_s, a_s),(j_s, b_s)})$ records the number 
of each type of sign which appears, see Figure \ref{P3mult}.  

\begin{figure}[htbp]
\centering
\includegraphics[scale = 0.4]{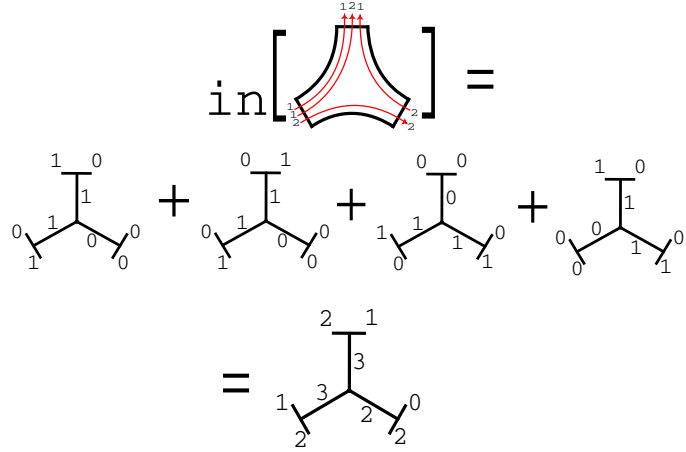}
\caption{Multiplying initial terms in $P_3$.}
\label{P3mult}
\end{figure}

\subsection{$\Gamma-$tensors}\label{gammatensors}

 For a graph $\Gamma,$ we introduce a combinatorial object called an abstract $\Gamma$-tensor $V(S, \phi)$ (see Figure \ref{pathglue}), this is the following information (see also \cite[Section 8]{M15}): 

\begin{enumerate}
\item A set of reduced directed paths $S_{v}$ in each link $\Gamma_v \subset \Gamma$.\\
\item For each edge $e \in E(\Gamma)$ with $\delta(e)= (v, u)$, a direction preserving bijection $\phi_e: S_{v, e} \to S_{u, e}$ between the paths through $e$ in $S_v$ and $S_u$.\\ 
\end{enumerate}

A marking $A$ of an abstract $\Gamma-$tensor is an assignment $A: S_{v} \to \{1, 2\}^2$ to the end points of each path such that the label on meeting endpoints of two paths connected by one of the maps $\phi_e$ are different. When a path $p_{ij} \in S_v$ is marked $A(i) = a, A(j) = b$, it can be viewed as a Pl\"ucker generator $p_{(i, a),(j,b)} \in \C[M_{\epsilon(v)}(SL_2(\C))].$ In this way an abstract $\Gamma$ tensor $V(S, \phi)$ with  $A$ defines a monomial $\mathcal{V}(S, \phi, A) \in \C[\prod_{v \in V(\Gamma)}M_{\epsilon(v)}(SL_2(\C))]$.

\begin{figure}[htbp]
\centering
\includegraphics[scale = 0.4]{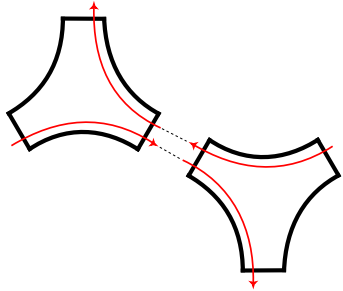}
\caption{An abstract $\Gamma$-tensor}
\label{pathglue}
\end{figure}

By \cite{M15}, Proposition $8.5$ (also see Proposition \ref{tracetensor} below), we obtain an invariant form $\mathcal{V}(S, \phi) \in \C[M_{\Gamma}(SL_2(\C))]$ $\subset \C[\prod_{v \in V(\Gamma)}M_{\epsilon(v)}(SL_2(\C))]$ by taking a signed sum over all $2^{E(\Gamma)}$ possible markings. 

\begin{equation}
\mathcal{V}(S, \phi) = \sum_{A} (-1)^{\rho(A)} \mathcal{V}(S, \phi, A)\\
\end{equation}

\noindent
Here $(-1)^{\rho(A)}$ is a product of signs defined by $\phi_e$ and the marking $A$ at an edge. For each identification
of endpoints $\phi_e(i) = j$, there is a contribution $(-1)^{\rho(\phi_e, i, j)}$, which is positive if $a(i) = 1, a(j) = 2$ and $i$ is outgoing, and negative if $a(i) =1, a(j) = 2$ and $i$ is incoming.  This assignment is illustrated on an example in Figure \ref{orientassign}.

\begin{figure}[htbp]
\centering
\includegraphics[scale = 0.4]{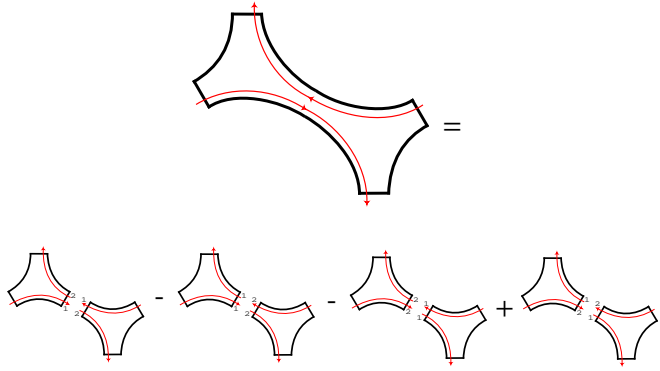}
\caption{A $\Gamma$ tensor obtained as a signed sum over sign assignments to an abstract $\Gamma$-tensor}
\label{orientassign}
\end{figure}

We let $\mathcal{S}(\Gamma) \subset \C[M_{\Gamma}(SL_2(\C))]$ be the set of all the forms $\mathcal{V}(S, \phi)$. 
Each abstract $\Gamma$-tensor $V(S, \phi)$ corresponds naturally to a weighting $a(S, \phi): E(\Gamma) \to \Z_{\geq 0}$, where $a(S, \phi)(e)$ is the number of paths passing through $e$.  The next proposition relates $\mathcal{V}(S, \phi)$ to the spin element $\Phi_{a(S, \phi)} \in \C[M_{\Gamma}(SL_2(\C))]$  when $\Gamma$ is trivalent.

\begin{proposition}\label{gammainitial}
Let $\Gamma$ be trivalent, then the initial form $in(\mathcal{V}(S, \phi))$ of a $\Gamma-$tensor is a non-zero multiple of $\Phi_{a(S, \phi)}.$ 
\end{proposition}

\begin{proof}
We observe that the initial form $in(\mathcal{V}(S, \phi)) \in \C[M_{\Gamma}(SL_2(\C)^c)]$ can be computed in the algebra $\C[M_{3}(SL_2(\C))^{V(\Gamma)}]$; the $SL_2(\C)^{E(\Gamma)}$ stability of the filtrations involved guarantees that the initial form of an invariant in $\C[M_{3}(SL_2(\C))^{V(\Gamma)}]$ is also invariant. The initial form of $in(\mathcal{V}(S, \phi, A))$ is a tensor
product (over $v \in V(\Gamma)$) of initial forms of the monomials defined by $S$ in $\C[M_{3}(SL_2(\C))]$.  Each of the $in(\mathcal{V}(S, \phi, A))$ lies in the $a(S, \phi)$ isotypical component of $\C[M_{3}(SL_2(\C))^{V(\Gamma)}]$, so it follows that $in(\mathcal{V}(S, \phi))$ equals the sum of these terms.   Initial forms with different markings on a path at an edge are linearly independent, as this is the case for the associated elements of $\C[P_3^{V(\Gamma)}]$. Since  $\mathcal{V}(S, \phi)$ always has exactly one term with all outgoing assignments $1$, this sum cannot vanish.  
\end{proof}

\begin{corollary}\label{gammaval}
A valuation $v_{\ell} \in C_{\Gamma},$ is computed on a $\Gamma-$tensor $\mathcal{V}(S, \phi) \in \C[\mathcal{X}(F_g, SL_2(\C))]$ as follows: 

\begin{equation}
v_{\ell}(\mathcal{V}(S, \phi)) = \sum_{e \in E(\Gamma)} \ell(e)a(S, \phi)(e).\\
\end{equation}

\end{corollary}

We have stated this corollary for trivalent $\Gamma$, but the general case can be recovered by considering
a weighting on $\Gamma$ as a weighting with $0$ entries on a trivalent cover $\tilde{\Gamma} \to \Gamma.$

\subsection{$\Gamma-$tensors and $F_g$}

For a marking $\phi: \Gamma_g \to \Gamma$, let $S(\Gamma, \phi) \subset \C[\mathcal{X}(F_g, SL_2(\C)]$ be the pullback of $S(\Gamma) \subset \C[M_{\Gamma}(SL_2(\C))]$.  We show that the sets $S(\Gamma, \phi) \subset \C[\mathcal{X}(F_g, SL_2(\C))]$ all coincide, and how to compute them as regular functions. From each connected abstract $\Gamma$ tensor $V(S, \phi)$ we obtain a reduced word $\omega(S, \phi) \in F_g$, this defines a bijection between abstract $\Gamma$-tensors and elements of $<F_g>$.  The $\Gamma-$tensor $\mathcal{V}(S, \phi)$ is then shown to coincide with the trace-word function $\tau_{\omega(S, \phi)}$ (see Figure \ref{pathword}).  We conclude by showing that for a valuation $v_{\Gamma, \ell, \phi} \in C_{\Gamma, \phi}$, $v_{\Gamma, \ell, \phi}(\tau_{\omega(S, P)})$ is the evaluation of the length function $d_{\omega(S, P)}$ on the metric graph $(\Gamma, \ell, \phi) \in \hat{O}(g)$.

In order to identify the $\mathcal{V}(S, \phi)$ as regular functions, we consider similar forms defined on trees. 
 A $\tree-$tensor $\mathcal{V}(S, \phi, B) \in \C[M_{\tree}(SL_2(\C))]$ is defined in the same way as a $\Gamma$-tensor, 
with the addition of sign information $B$ specified at the leaves. By construction, any $\mathcal{V}(S, \phi, B)$ factors
as a monomial in simple paths $\mathcal{V}(S_{\gamma}, \phi_{\gamma}, B_{\gamma})$, the next proposition shows
that each of these path tensors can be considered to be a Pl\"ucker generator from Subsection \ref{gammatensors}.

\begin{proposition}\label{g0}
The set of $\tree-$tensors $\mathcal{S}(\tree) \subset \C[M_{\tree}(SL_2(\C))]$ coincides with the set of Pl\"ucker monomials in $\C[M_{\mathcal{L}(\tree)}(SL_2(\C))]$ under the isomorphisms $\Phi_{\tree}$ and $\Psi_{\tree}.$
\end{proposition}

\begin{proof}
This is a straightforward calculation, also covered in Proposition $8.4$ of \cite{M15}. 
\end{proof}

We use this proposition to prove that $\mathcal{S}(\Gamma, \phi)$ coincides with the set $\mathcal{S}(\Gamma_g) \subset \C[\mathcal{X}(F_g, SL_2(\C))]$
for $\psi_{V, \tree}: \Gamma_g \to \Gamma$ and $\phi_{\tree}: \Gamma \to \Gamma_g$ a distinguished pair of maps corresponding to a spanning tree $\tree \subset \Gamma$.  Recall the maps $\Phi_{\tree}, \Psi_{\tree}$ and $\pi_{\tree}$ from Proposition \ref{spanningtreespace}.

\begin{proposition}\label{roseeq}
The set of $\Gamma-$tensors $\mathcal{S}(\Gamma) \subset \C[M_{\Gamma}(SL_2(\C))]$ coincides with the set of $\Gamma_g$-tensors in $\C[M_{\Gamma_g}(SL_2(\C))]$ under the isomorphisms $\Phi_{\tree}$ and $\Psi_{\tree}.$
\end{proposition}

\begin{proof}
This follows from the observation (see \cite{M15}, Proposition $8.5$) that any $\Gamma$ (respectively $\Gamma_g-$tensor) can be decomposed into a sum of $T(\tree, \Gamma)$-tensors, using the map $\pi_{\tree}.$

\begin{equation}
\mathcal{V}_{\Gamma}(S, \phi) = \sum_{B} \mathcal{V}_{T(\tree, \Gamma)}(S_{T(\tree, \Gamma)}, \phi_{T(\tree, \Gamma)}, B)(-1)^{\rho(B)}\\
\end{equation}

\noindent
Each of the components $\mathcal{V}_{T(\tree, \Gamma)}(S_{T(\tree, \Gamma)}, \phi_{T(\tree, \Gamma)}, B)$ is a monomial in Pl\"ucker generators by Proposition \ref{g0}.  This sum therefore amounts to an element of $\mathcal{S}(\Gamma_g)$.  This argument can 
then be run in reverse to give the other inclusion. 
\end{proof}

Next we show that the set $\mathcal{S}(\Gamma_g) \subset \C[M_{\Gamma_g}(SL_2(\C))]$ coincides with another set of distinguished regular functions on $M_{\Gamma_g}(SL_2(\C)) = \mathcal{X}(F_g, SL_2(\C))$. The trace-word function $\tau_{\omega}$ associated to a word $\omega \in F_g$ is the regular function that takes $(A_1, \ldots, A_g) \in \mathcal{X}(F_g, SL_2(\C))$ to the complex number $tr(\omega(A_1, \ldots, A_g))$.   Let $\mathcal{W}_g \subset \C[\mathcal{X}(F_g, SL_2(\C))]$ be the set of monomials in these functions, we will prove the following proposition with a series of lemmas.

\begin{proposition}\label{tracetensor}
The set $\mathcal{S}(\Gamma_g)$ coincides with $\mathcal{W}_g$. 
\end{proposition}

Notice that $\mathcal{W}_g$ is naturally closed under the action of the outer automorphism group $Out(F_g)$ on $\C[\mathcal{X}(F_g, SL_2(\C))]$, so Proposition \ref{tracetensor} implies that $\mathcal{S}(\Gamma, \phi) = \mathcal{S}(\Gamma_g) \subset \C[\mathcal{X}(F_g, SL_2(\C))]$ for $\phi$ any composition of an outer automorphism with a distinguished graph map.  Proposition \ref{hominv} then implies that this is the case for any marking $\phi: \Gamma_g \to \Gamma.$

First we show that these sets are bijective combinatorially.  Let $\mathcal{S}_{con}(\Gamma)$ be the set of connected abstract $\Gamma$-tensors. Fix a connected abstract $\Gamma_g$-tensor $V(S, \phi) \in \mathcal{S}_{con}(\Gamma_g)$.  We let $T(\Gamma_g, v)$ be the $2g$ tree obtained from the "spanning tree" $\{v\} \subset \Gamma_g$.  We label the leaves of $T(\Gamma_g, v)$ with the $2g$ indices $2i-1, 2i$ $1 \leq i \leq g$, where the $2i-1$ and $2i$ leaves are identified in the quotient map $T(\Gamma_g, v) \to \Gamma_g.$ We let $p_{(i_1,j_1)} \in S$ be a path in $V(S, \phi).$  The endpoint $j_1$ is connected with the starting point of the next path $p_{(i_2, j_2)},$ so $j_1 = 2a_1 -1, i_2 = 2a_1$ or $j_1 = 2a_1, i_2 = 2a_1-1$.   Continuing this way, we obtain a word $\omega(V(S, \phi), p_{(i_1, j_1)}) = x_{a_1}^{\epsilon_1}x_{a_2}^{\epsilon_2}\ldots x_{a_m}^{\epsilon_m}$, where the sign $\epsilon_i$ is determined by the rule $j_1 = 2a_1 -1, i_2 = 2a_1 = 2 \implies \epsilon_i = 1$, $j_1 = 2a_1, i_2 = 2a_1-1 \implies \epsilon_i = -1.$  By construction, this word is reduced, as $\Gamma_g$ tensors do not have back-tracks by definition.  Given a word $\omega$, we may reverse this recipe to obtain a $\Gamma_g-$tensor $V(S_{\omega}, \phi_{\omega}).$  All paths in $T(\Gamma_g, v)$ are determined by their endpoints, so we have a $1-1$ map from the set of connected $\Gamma_g$-tensors $V(S, \phi)$ with a choice of initial path  $p_{(i_1,j_1)} \in S$ to the set of reduced words in $F_g.$  Changing the initial path amounts to changing the word $\omega(V(S, \phi), p_{(i_1,j_1)})$ by a cyclic permutation, so we have proved the following lemma. 

\begin{lemma}
The set $\mathcal{S}_{con}(\Gamma_g)$ is in bijection with cyclic equivalence classes of reduced words in $F_g.$
\end{lemma}

By Proposition \ref{roseeq}, the same is true for the connected $\Gamma$ tensors $\mathcal{S}_{con}(\Gamma)$ of any graph $\Gamma.$ 
The recipe to compute $\omega(V(S, \phi), p_1) \in F_g$ is almost identical in the general case, except one replaces a simple path $p_{(i_s,j_s)}$ with the shortest path $p_s$ in the chosen spanning tree $\tree \subset \Gamma.$

Now we show that any connected $\Gamma_g-$tensor can be ``untangled," in the sense that it is realized as the pullback of a fixed $\Gamma_{g'}$ tensor by a map $\phi_*: M_{\Gamma_{g'}}(SL_2(\C)) \to M_{\Gamma_g}(SL_2(\C))$, where $\Gamma_{g'}$ has a different genus. 

\begin{definition}
We let $\mathcal{V}_g \in \C[M_{\Gamma_g}(SL_2(\C))] = \C[SL_2(\C) \ql SL_2(\C)^{2g} \q SL_2(\C)^g]$ be the $\Gamma_g$-tensor defined 
by joining the paths $p_{(2g, 1)} \to p_{(2, 3)} \to \ldots \to p_{(2g-2, 2g-1)}$, as in Figure \ref{rose}. 
\end{definition}

\begin{figure}[htbp]
\centering
\includegraphics[scale = 0.6]{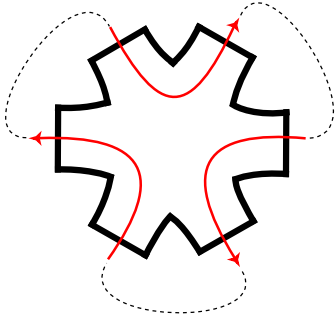}
\caption{The tensor $\mathcal{V}_{3}$}
\label{rose}
\end{figure}

\noindent
We fix a reduced word $\omega \in F_g$, and we let $n = |\omega|$.  There is a natural map of $\phi_{\omega}: \mathcal{X}(F_g, SL_2(\C)) \to \mathcal{X}(F_n, SL_2(\C))$ defined by sending $(A_1, \ldots, A_g)$ to the entries of $\omega$ in order, $(\omega_1(\vec{A}), \ldots, \omega_n(\vec{A)}).$  Recall that $\mathcal{X}(F_g, SL_2(\C))$ is the quotient $SL_2(\C) \ql SL_2(\C)^{2g} \q SL_2(\C)^g = M_{2g}(SL_2(\C))\q SL_2(\C)^g$, by way of the map $\pi_{0, g}(k_1, h_1, \ldots, k_g, h_g) = (k_1h_1^{-1}, \ldots, k_gh_g^{-1}).$  We define a map $\phi_{\omega}'$ by sending the tuple $(k_1, h_1, \ldots, k_g, h_g)$ to $(\omega_1(\vec{k}), \omega_1(\vec{h}), \ldots, \omega_n(\vec{k}), \omega_n(\vec{h}))$, this commutes with $\phi_{\omega}$ under the isomorphisms $\pi_{0, g}, \pi_{0, n}.$  We will see where $\mathcal{V}_g$ goes under $\phi_{\omega}^*.$ 

\begin{lemma}
\begin{equation}
\phi_{\omega}^*(\mathcal{V}_n) = \mathcal{V}(S_{\omega}, \phi_{\omega})\\
\end{equation}
\end{lemma}

\begin{proof}
When the Pl\"ucker element $p_{(2i, a), (2i+1, b)}$ is evaluated on $\phi_{\omega}'(k_1, h_1, \ldots, k_g, h_g)$ it gives the determinant of the $a$ column of $\omega_i(\vec{h})$ and the $b$ column of $\omega_{i+1}(\vec{k}).$  The pullback $\phi_{\omega}^*(\mathcal{V}_n)$ therefore coincides with $\mathcal{V}(S_{\omega}, \phi_{\omega})$ by definition.     
\end{proof}

We now introduce the model trace-word $\tau_g = tr(A_1A_2 \ldots A_g)$.  By the previous lemma, if we can show that $\tau_n = \mathcal{V}_n \in \C[\mathcal{X}(F_n, SL_2(\C))]$, we would have $\mathcal{V}(S_{\omega}, \phi_{\omega}) = \tau_{\omega} = tr(\omega(\vec{A})),$ and Proposition \ref{tracetensor} would follow. 

\begin{lemma}\label{rosetrace}

\begin{equation}
\tau_n = \mathcal{V}_n \in \C[\mathcal{X}(F_n, SL_2(\C))]\\
\end{equation}
\end{lemma}

\begin{proof}
We employ the following commutative diagram:

\footnotesize 
$$
\begin{CD}
 ([SL_2(\C)\times SL_2(\C)]\q SL_2(\C) \times  \big{[}[SL_2(\C) \times SL_2(\C)]\q SL_2(\C)\big{]}^{g-1} @>>> SL_2(\C) \times SL_2(\C)^{g-1}\\
@VVV @VVV\\
([SL_2(\C) \times SL_2(\C)]\q SL_2(\C))^{g-1} @>>> SL_2(\C)^{g-1}\\
\end{CD}
$$

\normalsize

$$
\begin{CD}
 (k_1, h_1), (k_2, h_2), \ldots, (k_g, h_g) @>>> (k_1h_1^{-1}, \ldots, k_gh_g^{-1})\\
@VVV @VVV\\
(k_1h_1^{-1}, h_2k_2^{-1}), \ldots, (k_g, h_g) @>>> (k_1h_1^{-1}k_2h_2^{-1}, \ldots, k_gh_g^{-1}).\\
\end{CD}
$$

The horizontal arrows in this diagram give the isomorphism between the spaces $M_{2g}(SL_2(\C))\q SL_2(\C)^g$ and $\mathcal{X}(F_g, SL_2(\C))$. 
The vertical arrow on the right is the matrix composition operation $(A_1, \ldots, A_g) \to (A_1A_2, \ldots, A_g) \in SL_2(\C)^{g-1}.$  The vertical arrow on the left lifts this to $([SL_2(\C) \times SL_2(\C)]\q SL_2(\C))^{g}$, where we can talk about $\Gamma-$tensors.  

Notice that $\tau_{g-1}$ pulls back to $\tau_g$ under the right vertical arrow.  It follows that if we can show that $\tau_1 = \mathcal{V}_1 \in \C[SL_2(\C)]^{SL_2(\C)},$ and that $\mathcal{V}_{g-1}$ pulls back to $\mathcal{V}_g$ under the left vertical arrow, we have the lemma.  The first assertion holds by a straightforward calculation. The second assertion follows from the fact that $\mathcal{V}_{g-1} = \sum (p_{(2g, a_{2g}), (1, a_1)}p_{(4, a_4), (5, a_5)}\cdots)$ evaluated at $(k_1h_1^{-1}, h_2k_2^{-1}), \ldots, (k_g, h_g)$ gives the same answer as $\mathcal{V}_{g} = \sum (p_{(2g, a_{2g}), (1, a_1)}p_{(2, a_2),(3,a_3)} p_{(4, a_4), (5, a_5)}\cdots)$ evaluated at $(k_1, h_1), (k_2, h_2), \ldots, (k_g, h_g)$. This can be checked by direct calculation on the case $g = 5$. 
\end{proof}

\begin{figure}[htbp]
\centering
\includegraphics[scale = 0.4]{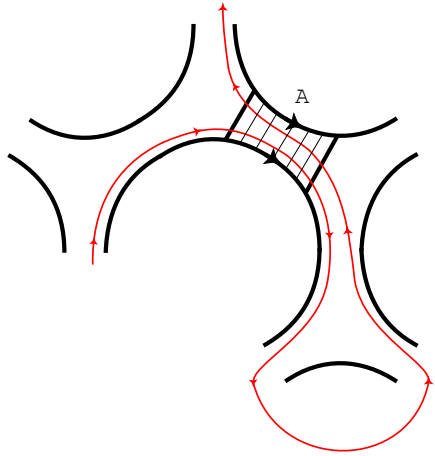}
\caption{A path representing $tr(\ldots A \ldots A^{-1} \ldots)$}
\label{pathword}
\end{figure}

For each metric graph $(\Gamma, \ell),$ with a marking $\phi: \Gamma_g \to \Gamma,$ with corresponding valuation $v_{ \Gamma, \ell, \phi}: \C[\mathcal{X}(F_g, SL_2(\C))] \to \Z \cup\{-\infty\}$, the following holds by Corollary \ref{gammaval}: 

\begin{equation}\label{lengthequation}
 v_{\Gamma, \ell, \phi}(\tau_{\omega}) = v_{\Gamma, \ell, \phi}(\mathcal{V}(S_{\omega}, \phi_{\omega})) = d_{\omega}(\Gamma, \ell, \phi).\\
\end{equation}

\subsection{Proof of Theorem \ref{mainvaluation}}\label{proofofmainvaluation}

Now we finish the proof of Theorem \ref{mainvaluation} by showing that $\Sigma: \hat{O}(g) \to \mathcal{X}(F_g, SL_2(\C))^{an}$, $\Sigma(\Gamma, \ell, \phi) = v_{\Gamma, \ell, \phi}$ gives an embedding.  Outer space $O(g)$ is the closed subset of $\hat{O}(g)$ of marked metric graphs $(\Gamma, \ell, \phi)$ with volume $1$, this also gives an embedding of outer space into $\mathcal{X}(F_g, SL_2(\C))^{an}$.

  The topology on $\mathcal{X}(F_g, SL_2(\C))^{an}$ is generated by the the sets $ev_f^{-1}(U)$ where the $U \subset \R$ are open intervals. This is similar to the topology on $\hat{O}(g)$, which is generated by the sets $d_w^{-1}(U)$ (\cite{CV}).  Each length function on $\hat{O}(g)$ extends to a continuous function on $\mathcal{X}(F_g, SL_2(\C))^{an}$ by Equation \ref{lengthequation}, so it follows that any open set in $\hat{O}(g)$ is induced from its inclusion into $\mathcal{X}(F_g, SL_2(\C))^{an}$. It remains to check that any evaluation function is continuous on $\hat{O}(g),$ as this would imply that no new open sets are induced from $f \in \C[\mathcal{X}(F_g, SL_2(\C))]$, $f \neq \tau_{\omega}.$   

Note that $f$ can be expanded in a unique way into $(\Gamma, \phi)$ spin diagrams $\Phi_a \in \mathcal{R}(\Gamma)$, and that the value $v_{\Gamma, \ell, \phi}(f)$ is obtained on one or more of these diagrams.  Now we may replace each spin diagram function $\Phi_a$ with a lower-triangular expansion into trace-word functions $\Phi_a = \sum \tau_{w_{a,i}}$, with $v_{\Gamma, \ell, \phi}(\Phi_a) = v_{\Gamma, \ell, \phi}(\tau_{w_{a, 1}})$, the value of the leading term.   The equivalence class of $\Phi_a$ equals that of $\tau_{w_{a, 1}}$ in the associated graded algebra $\C[P_{\Gamma}]$ of any interior point of $C_{\Gamma, \phi}$; it follows that $v_{\Gamma, \ell, \phi}(f) = max_{C_a \neq 0}\{ \ldots, v_{\Gamma, \ell, \phi}(\Phi_a)\ldots \}$ $ = max_{C_a \neq 0}\{ \ldots, v_{\Gamma, \ell, \phi}(\tau_{w_{a, 1}})\ldots \}$ for any $\ell \in C_{\Gamma, \phi}$.  This shows that $ev_f$ is the max of a finite number of continuous functions over the simplicial cone $C_{\Gamma, \phi}$.  As the function $ev_f$ is continuous when restricted to any $C_{\Gamma, \phi}$, it must be continuous on $\hat{O}(g)$ by the Pasting Lemma.  This proves Theorem \ref{mainvaluation}.

\subsection{The tropical variety $\trop(I_{2, g})$ and $\Upsilon_g$}

We refer the reader to the book of Maclagan and Sturmfels \cite{MSt} for background on tropical geometry.  We will need the following
result (see \cite{P}) relating the analytification $X^{an}$ of an affine variety to the tropical variety $\trop(I)$ of an ideal which cuts $X$ out of affine space. 

\begin{theorem}[Payne]\label{Payne}
Let $\pi: k[x_1, \ldots, x_n] \to k[X]$ be a presentation of the coordinate ring of $X$ with $I = Ker(\pi)$, then the image of the map $(ev_{\pi(x_1)}, \ldots, ev_{\pi(x_n)}): X^{an} \to \R^n$ is the tropical variety $\trop(I)$. 
\end{theorem}

We fix a set of generators $x_1, \ldots, x_g \in F_g$ and consider the set $S_{2, g}$ of trace word functions where no $x_i$ appears more than twice, including inverses and multiplicity.  For a parameter set $Y_{2, g}$ in bijection with $S_{2, g}$ we let $I_{2, g} \subset \C[Y_{2, g}]$ be the ideal of forms which vanishes on the associated map $\C[Y_{2, g}] \to \C[\mathcal{X}(F_g, SL_2(\C))]$. The choice of generators $x_1, \ldots, x_g$ determines an embedding of the space of spanned metric graphs $\Upsilon_g$ into $ \hat{O}_g$ (Subsection \ref{spanneddef}).   We finish this section with a proof of the following theorem. 

\begin{theorem}\label{Khovanskii}
The map $(\ldots ev_{\tau_w}\ldots)_{w \in S_{2, g}}: X(F_g, SL_2(\C))^{an} \to \trop(I_{2, g})$ is $1-1$ on $\Upsilon_g \subset \hat{O}_g$.
\end{theorem}

The general principle at work in this statement is encapsulated in the following lemma.

\begin{lemma}\label{Ggen}
Let $G \subset \C[\mathcal{X}(F_g, SL_2(\C))]$ be any set so that the equivalence classes of the elements $f \in X$ generate $gr_{v_{\Gamma, \tree, \ell}}(\C[\mathcal{X}(F_g, SL_2(\C))]) \cong \C[P_{\Gamma}]$ for all $(\Gamma, \tree, \ell) \in \Upsilon_g$, then the map $(\ldots ev_{f} \ldots)_{f \in G}$ is $1-1$ on $\Upsilon_g$.
\end{lemma}

\begin{proof}
A choice $(\Gamma, \tree)$ determines a basis $\mathcal{R}(\Gamma, \tree) \subset \C[\mathcal{X}(F_g, SL_2(\C))]$ of spin diagrams.  Each $f \in G$ has an expansion $f = \sum C_a \Phi_a$ in these diagrams, and the equivalence class of $f$ in $gr_{v_{\ell}}(\C[\mathcal{X}(F_g, SL_2(\C))])$ for $\ell$ an interior point of $C_{\Gamma, \tree} \subset \Upsilon_g$ is the initial form of $f$ with respect to $v_{\ell}$.   By assumption the $a$ of these initial forms generate $P_{\Gamma}$, and by Proposition \ref{spinadapt} the values on these generators determine all outputs of the valuation $v_{\ell}$.  It follows that if $v_, v_2 \in \Upsilon_g$ have the same values on every element of $G$ then they must be the same valuation. 
\end{proof}

Theorem \ref{Khovanskii} is a consequence of Lemma \ref{Ggen} and the following proposition, which is \cite[Theorem 7.2]{M14}.

\begin{proposition}
The affine semigroup $P_{\Gamma}$ is generated by those $a: E(\Gamma) \to \Z_{\geq 0}$ with $a(e) \leq 2$. 
\end{proposition}

\noindent
The set $S_{2, g}$ is said to be a Khovanskii basis of $\C[\mathcal{X}(F_g, SL_2(\C))]$ with respect to each valuation $v_{\ell}$ with $\ell \in \Upsilon_g$ because the equivalence classes of these functions generate each associated graded algebra of each $v_{\ell}$.  Results from \cite{M14} also imply that $S_{2, g}$ is a Khovanskii basis for a collection of maximal rank ($= dim(\mathcal{X}(F_g, SL_2(\C)))$) valuations $v_{\Gamma, \prec}$ (see Section \ref{NOKsection}).  This terminology is used in honor of the contributions of Askold Khovanskii to the intersection of combinatorics and algebraic geometry.

\section{The integrable Hamiltonian system in $\mathcal{X}(F_g, SL_2(\C))$ associated to $\Gamma$ }\label{Hamiltonian}

In this section we construct the spaces $M_{\Gamma}(SL_2(\C))$ and  $M_{\Gamma}(SL_2(\C)^c)$ as symplectic reductions. We show that
each distinguished map $\phi_{\tree}: \Gamma \to \Gamma_g$ associated to a spanning tree $\tree$ gives a symplectomorphism
$\phi_{\tree}^*: \mathcal{X}(F_g, SL_2(\C)) \to M_{\Gamma}(SL_2(\C)).$ Then we construct a surjective, continuous contraction map
$\Xi_{\Gamma}: M_{\Gamma}(SL_2(\C)) \to M_{\Gamma}(SL_2(\C)^c)$, and we show that this map is a symplectomorphism on a dense, 
open subspace $M^o_{\Gamma}(SL_2(\C)) \subset M_{\Gamma}(SL_2(\C))$.  The compact part $\mathbb{T}_{\Gamma}$ of the defining torus $T_{\Gamma}$ of the toric variety $M_{\Gamma}(SL_2(\C)^c)$ is used to define an integrable system on $M^o_{\Gamma}(SL_2(\C))$ with momentum map $\mu_{\mathbb{T}_{\Gamma}} \circ \Xi_{\Gamma}.$  We show that the fibers of the contraction map $\Xi_{\Gamma}$ are all connected, and that the fiber over the origin is the compact character variety $\mathcal{X}(F_g, SU(2)).$

\subsection{$M_{\Gamma}(SL_2(\C))$ and $M_{\Gamma}(SL_2(\C)^c)$ as symplectic reductions.}

The product spaces $SL_2(\C)^{E(\Gamma)}$ and $[SL_2(\C)^c]^{E(\Gamma)}$ have natural K\"ahler structures as subspaces of  $M_{2 \times 2}(\C)^{E(\Gamma)}$, and the actions of $SU(2)^{V(\Gamma)}$ (as a subgroup of $SL_2(\C)^{V(\Gamma)}$) on these spaces are Hamiltonian.  The work of Kirwan \cite{Kir} and Kempf-Ness \cite{KN} shows that the reductions $SL_2(\C)^{E(\Gamma)}/_0 SU(2)^{V(\Gamma)}$, $[SL_2(\C)^c]^{E(\Gamma)}/_0 SU(2)^{V(\Gamma)}$ carry stratified K\"ahler structures, and can be identified with the spaces $M_{\Gamma}(SL_2(\C))$ and $M_{\Gamma}(SL_2(\C)^c)$ respectively.  We show that the symplectic structure on $M_{\Gamma}(SL_2(\C))$ obtained by this reduction procedure does not depend on $\Gamma.$

Let $\tree$ be a directed tree with one internal edge $e \in E(\tree)$ and endpoints $\delta(e) =(u, v)$, and let $\tree'$ be the tree obtained from $\tree$ by collapsing $e$ to a single vertex $w$. We let $S_1 \cup S_2 = \mathcal{L}(\tree)$ be the partition of the leaves of $\tree$ defined by $e.$

\begin{lemma}\label{treelemma}
There are symplectomorphisms $\phi: M_{\tree}(SL_2(\C)) \to M_{\tree'}(SL_2(\C)), \psi:M_{\tree'}(SL_2(\C)) \to M_{\tree}(SL_2(\C)).$
Furthermore, both of these maps are isomorphisms of $SU(2)^{\mathcal{L}(\tree)}$ Hamiltonian spaces. 
\end{lemma}

\begin{proof}
Recall that $SL_2(\C)$ can be identified with $T^*(SU(2))$ as a Hamiltonian $SU(2) \times SU(2)$ space.  For any Hamiltonian $SU(2)$ space $X$ there are isomorphisms  $\bar{\phi}: X \to SU(2) \backslash_0 [SL_2(\C) \times X]$, $\bar{\psi}: SU(2) \backslash_0[SL_2(\C) \times X] \to X$ computed as follows (\cite[proof of Lemma 4.8]{GJS}):

\begin{equation}
\bar{\phi}(x) = ([Id, -\mu(x)], x) \ \ \ \ \bar{\psi}([k, v], x) = k^{-1}x\\  
\end{equation}

Furthermore, these maps intertwine the left $SU(2)$ action on $SL_2(\C) \times X$ with the $SU(2)$ action on $X$, and if $X$ carries a Hamiltonian $K$ action for some compact Lie group $K$, this action is carried through $\bar{\phi}$. 

The space $M_{\tree'}(SL_2(\C))$ is an $SU(2) \times SU(2)$ reduction of $SL_2(\C)^{S_1}\times SL_2(\C) \times SL_2(\C)^{S_2}$.  By setting $X = SL_2(\C)^{S_2}$, $\bar{\phi}$ and $\bar{\psi}$ identify $SU(2) \backslash_0 [SL_2(\C) \times SL_2(\C)^{S_2}]$ with $SL_2(\C)^{S_2}$.  Under this isomorphism the right $SU(2)$ action on $SL_2(\C)$ is identified with the left diagonal action on $SL_2(\C)^{S_2}$.  After taking the product with $SL_2(\C)^{S_1}$ and reducing by $SU(2)$, $\bar{\phi}$ and $\bar{\psi}$ descend to give isomorphisms $\phi$ and $\psi$ between $M_{\tree'}(SL_2(\C))$ and the diagonal left $SU(2)$ reduction of $SL_2(\C)^{\mathcal{L}(\tree)} = SL_2(\C)^{S_1 \cup S_2}$.  This latter space is $M_{\tree}(SL_2(\C))$. 
\end{proof}

\begin{proposition}
Let $\tree \subset \Gamma$ be a spanning tree, with corresponding markings $\phi_{\tree, V}: \Gamma_g \to \Gamma,$ $\psi_{\tree}: \Gamma \to \Gamma_g.$ This induces maps $\Phi_{\tree}: M_{\Gamma}(SL_2(\C)) \to M_{\Gamma_g}(SL_2(\C))$, $\Psi_{\tree}: M_{\Gamma_g}(SL_2(\C)) \to M_{\Gamma}(SL_2(\C))$ which are symplectomorphisms.
\end{proposition}

\begin{proof}
This follows by repeatedly applying Lemma \ref{treelemma} to the edges of $\tree.$
\end{proof}

\subsection{The Hamiltonian $\mathbb{T}_{\Gamma}$-space $M_{\Gamma}(SL_2(\C)^c)$}

By Proposition \ref{contractiontoricvariety}, the space $M_{\Gamma}(SL_2(\C)^c)$ is the affine toric variety associated to the convex polyhedral cone $\mathcal{P}_{\Gamma}$. We describe a Hamiltonian action by a torus $\mathbb{T}_{\Gamma}$ on this space with momentum image equal to $\mathcal{P}_{\Gamma}$.  The space $SL_2(\C)^c$ is an $S^1$ reduction of $SL_2(\C)\q U_- \times U_+\ql SL_2(\C)$, considered as a product of imploded cotangent bundles of $SU(2)$, and the momentum map $\mu: SL_2(\C)^c \to \R$ of the residual $S^1$ action takes a class $([k, w][h, v]) $ to $|w| = |v|$.   The $SU(2)\times SU(2)$ and $S^1$ actions commute on $SL_2(\C)^c$, so it follows that we may realize $M_{\Gamma}(SL_2(\C)^c)$ as an $SU(2)^{V(\Gamma)}\times (S^1)^{E(\Gamma)}$ reduction of $[SL_2(\C)\q U_- \times U_+\ql SL_2(\C)]^{E(\Gamma)}.$

By first reducing with respect to $SU(2)^{V(\Gamma)}$, we are led to consider the spaces $SU(2) {}_0\backslash [SL_2(\C)\q U]^3.$
The momentum level $0$ set for the action of $SU(2)$ on $[SL_2(\C)\q U]^3$ is the space $\mu_{SU(2)}^{-1}(0)$ of imploded triples $[k_1, r_1\varrho][k_2, r_2\varrho][k_3, r_3\varrho]$ such that

\begin{equation}
r_1 k_1\circ\varrho + r_2k_2\circ\varrho + r_3k_3\varrho = 0\\
\end{equation}

The quotient space $SU(2) \backslash \mu_{SU(2)}^{-1}(0)$ is known as the space of spin-framed triangles in $\R^3 \cong su(2)^*$, see \cite{HMM}.  It is known that $SU(2) \backslash \mu_{SU(2)}^{-1}(0)$ is symplectomorphic to the affine toric manifold $\bigwedge^2(\C^3)$.
Notice that the action of $(S^1)^3$ on this space has momentum map 

\begin{equation}
\mu_{(S^1)^3}([k_1, r_1\varrho][k_2, r_2\varrho][k_3, r_3\varrho]) = (r_1, r_2, r_3).\\
\end{equation}

\noindent
The momentum image is therefore the polyhedral cone $\Delta_3$ of non-negative real triples $(r_1, r_2, r_3)$ which can be the sides of a triangle.

Now we may realize $M_{\Gamma}(SL_2(\C)^c)$ as an $(S^1)^{E(\Gamma)}$ reduction of the affine toric symplectic manifold $\prod_{v \in V(\Gamma)} SU(2) \backslash \mu_{SU(2)}^{-1}(0)$.  The torus $(S^1)^{2E(\Gamma)}$ acts on the product $\prod_{v \in V(\Gamma)}SU(2) {}_0\backslash [SL_2(\C)\q U]^3$; where there is a copy of $(S^1)^2$  for each edge $e \in E(\Gamma)$.  The torus $(S^1)^{E(\Gamma)}$ is a product of the anti-diagonally embedded copies of the circle: $S^1 \subset (S^1)^2$, $t \to (t, t^{-1})$.  Following the definition of $SL_2(\C)^c,$ the momentum map of the copy of $S^1$ associated to an edge $e \in E(\Gamma)$ is the difference in lengths of the two imploded coordinates associated to $e$ in this product.  We let $\mathbb{T}_{\Gamma} \subset (S^1)^{2E(\Gamma)}$ be the product of circles $S^1 \subset (S^1)^2$ embedded by $t \to (t, 1)$; as remarked above, the component associated a single edge $e \in E(\Gamma)$ has momentum map $\mu: SL_2(\C)^c \to \R$.

\begin{proposition}
The space $M_{\Gamma}(SL_2(\C)^c)$ is a Hamiltonian toric space, and the momentum image of the torus $\mathbb{T}_{\Gamma}$
is the polyhedral cone $\mathcal{P}_{\Gamma}.$
\end{proposition}

\begin{proof}
The space $\prod_{v \in V(\Gamma)} SU(2) \backslash \mu_{SU(2)}^{-1}(0)$ has a Hamiltonian action of $(S^1)^{2E(\Gamma)}$, with momentum image $\prod_{v \in V(\Gamma)} \Delta_3.$  In reducing by $(S^1)^{E(\Gamma)}$, we pass to the subquotient
with residual momentum image the subcone of $\prod_{v \in V(\Gamma)} \Delta_3$ where two components corresponding to the same edge $e \in E(\Gamma)$ have the same value, this is $\mathcal{P}_{\Gamma}$ by definition.  Furthermore, this is the momentum map of the residual action by the torus $\mathbb{T}_{\Gamma}$.
\end{proof}

\subsection{The $\Gamma$-collapsing map}

Recall that the collapsing map $\Xi: SL_2(\C) \to SL_2(\C)^c$ is a surjective, 
continuous map of Hamiltonian $SU(2) \times SU(2)$ spaces, and that it intertwines the left and right $SU(2)$ momentum mappings on these spaces. We let $\Xi^{E(\Gamma)}: SL_2(\C)^{E(\Gamma)} \to [SL_2(\C)^c]^{E(\Gamma)}$ be the product map, which
is likewise a map of Hamiltonian $SU(2)^{V(\Gamma)}$ spaces, see Figures \ref{TriangleDiagram} and \ref{TriangleDiagram2}.   It follows that $\Xi^{E(\Gamma)}$ restricts to a surjective, continuous map on the $0$-level sets of the
map(s) $\mu_{SU(2)^{V(\Gamma)}}$.  We let $\Xi_{\Gamma}$ be the map on the quotient by $SU(2)^{V(\Gamma)}.$

\begin{figure}[htbp]

$$
\begin{xy}
(0, -3)*{su(2)^*\times su(2)^*} = "A";
(3.5, 0)*{} = "A1";	
(-3.5, 0)*{} = "A2";
(25, 30)*{SL_2(\C)^c} = "B";
(20,25)*{} = "B1";
(15, 30)*{} = "B2";
(-25, 30)*{SL_2(\C)} = "C";
(-20,25)*{} = "C1";
(-15, 30)*{} = "C2"; 
(0, 32)*{\Xi};
(-20, 14)*{\mu_L \times \mu_R};
(20, 14)*{\mu_L\times \mu_R};
{\ar@{>}"B1"; "A1"};
{\ar@{>}"C1"; "A2"};
{\ar@{>}"C2"; "B2"};
\end{xy}
$$\\
\caption{$\Xi$ intertwines the momentum maps.}
\label{TriangleDiagram}
\end{figure}
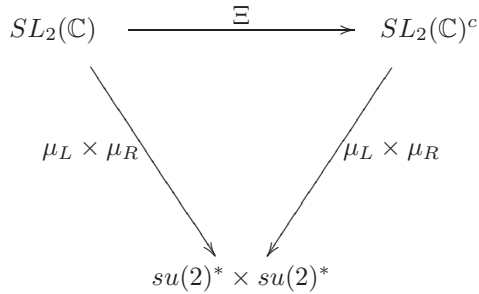

\begin{figure}[htbp]

$$
\begin{xy}
(0, -3)*{[su(2)^*]^{V(\Gamma)}} = "A";
(3.5, 0)*{} = "A1";
(-3.5, 0)*{} = "A2";
(26, 30)*{[SL_2(\C)^c]^{E(\Gamma)}} = "B";
(20,25)*{} = "B1";
(15, 30)*{} = "B2";
(-25, 30)*{SL_2(\C)^{E(\Gamma)}} = "C";
(-20,25)*{} = "C1";
(-15, 30)*{} = "C2"; 
(0, 32)*{\Xi^{E(\Gamma)}};
(-22, 14)*{\mu_{SU(2)^{V(\Gamma)}}};
(22, 14)*{\mu_{SU(2)^{V(\Gamma)}}};
{\ar@{>}"B1"; "A1"};
{\ar@{>}"C1"; "A2"};
{\ar@{>}"C2"; "B2"};
\end{xy}
$$\\
\caption{$\Xi^{E(\Gamma)}$ also intertwines the momentum maps.}
\label{TriangleDiagram2}
\end{figure}

\begin{equation}
\Xi_{\Gamma}: \mathcal{X}(F_g, SL_2(\C)) \cong M_{\Gamma}(SL_2(\C)) \to M_{\Gamma}(SL_2(\C)^c)\\
\end{equation}

Following Subsection \ref{themapxi}, we consider the family $det^{\Gamma}: M_{2\times 2}(\C)^{E(\Gamma)} \to \C^{E(\Gamma)}$.  This family is the Vinberg enveloping monoid of the group $SL_2(\C)^{E(\Gamma)}$.  By \cite[Section 6]{HMM2}, for any assignment of dominant $SL_2(\C)$ coweights $a: E(\Gamma) \to \Z_{> 0}$ (namely a choice of a point in $C_{\Gamma}$) there is an $SU(2)^{2E(\Gamma)}$-equivariant gradient Hamiltonian flow $V_{\pi, a}$ on the flat $\C-$family $E_a$ obtained by base change under the map $\phi_{a}: \C \to \C^{E(\Gamma)}$, $z \to (\ldots, z^{a(e)}, \ldots)$:

$$
\begin{CD}
E_a = M_{2\times2}(\C)\times_{\phi_a, det^{\Gamma}}\C @>>> M_{2\times 2}(\C)^{E(\Gamma)}\\
@VVV @V det^{\Gamma}VV\\
\C @>\phi_a>> \C^{E(\Gamma)}\\
\end{CD}
$$

\noindent
By \cite[Section 6]{HMM2}, this flow is completed by the continuous map $\Xi^{E(\Gamma)}: SL_2(\C)^{E(\Gamma)} \to [SL_2(\C)^c]^{E(\Gamma)}$. As this map is $SU(2)^{V(\Gamma)}$-equivarient, it descends to the symplectic reductions, where it must agree with $\Xi_{\Gamma}$.

\subsection{The integrable system in $\mathcal{X}(F_g, SL_2(\C))$}

By Proposition \ref{sl2xi}, $\Xi^{E(\Gamma)}: SL_2(\C)^{E(\Gamma)} \to [SL_2(\C)^c]^{E(\Gamma)}$ is a symplectomorphism on $SL_2(\C)_o^{E(\Gamma)}$; the latter being the set of points $(\ldots, [k_e, v_e], \ldots)$ with $v_e \neq 0$ for all $e \in E(\Gamma)$.   The product  $SL_2(\C)_o^{E(\Gamma)}$ is preserved by the $SU(2)^{V(\Gamma)}$ action, in particular it inherits an $SU(2)^{V(\Gamma)}$ Hamiltonian structure from $SL_2(\C)^{E(\Gamma)}$, and $\Xi^{E(\Gamma)}$ restricts to give an isomorphism of Hamiltonian $SU(2)^{V(\Gamma)}$ spaces onto the image in $[SL_2(\C)^c]^{E(\Gamma)}$.  

\begin{definition}
The subspace  $M_{\Gamma}^o(SL_2(\C)) \subset M_{\Gamma}(, SL_2(\C))$ is the space of those points $(\ldots [k_e, v_e], \ldots)$ with $v_e \neq 0$.
\end{definition}

\begin{proposition}
The map $\Xi_{\Gamma}: M_{\Gamma}(SL_2(\C)) \to M_{\Gamma}(SL_2(\C)^c)$ restricts to a symplectomorphism from $M_{\Gamma}^o(SL_2(\C))$ onto its image in $M_{\Gamma}(SL_2(\C)^c)$. 
\end{proposition}

\begin{proof}
This follows from the discussion above and the fact that $M_{\Gamma}^o(SL_2(\C))$ is the reduction at level $0$ of $SL_2(\C)_o^{E(\Gamma)}$ by $SU(2)^{V(\Gamma)}$.
\end{proof}

This completes the proof of Theorem \ref{mainsymplectic}.  The integrable $\mathbb{T}_{\Gamma} = [S^1]^{E(\Gamma)}$ action on $M_{\Gamma}(SL_2(\C)^c)$ also defines a Hamiltonian action on the image of $M_{\Gamma}^o(SL_2(\C))$. The momentum map for this action is computed as $\mu_{\mathbb{T}_{\Gamma}}\circ \Xi_{\Gamma}: \mathcal{X}(F_g, SL_2(\C)) \to \mathcal{P}_{\Gamma} \subset \R^{E(\Gamma)}$.  In particular, the Hamiltonians on $M_{\Gamma}^o(SL_2(\C))$ extend continuously to $\mathcal{X}(F_g, SL_2(\C)).$

We let $\Xi^{-1}_{\Gamma}: M_{\Gamma}^o(SL_2(\C)^c) \to M_{\Gamma}^o(SL_2(\C))$ be the inverse map to the restriction of $\Xi_{\Gamma} $ to $M_{\Gamma}^o(SL_2(\C)).$  The Hamiltonian action of $\mathbb{T}_{\Gamma}$ is computed through $\Xi^{-1}_{\Gamma}$ as follows. Let $(t_e) \in \mathbb{T}_{\Gamma}$, and $(\ldots, [k_e, v_e], \ldots) \in M_{\Gamma}(SL_2(\C)) = \mathcal{X}(F_g, SL_2(\C))$.

\begin{equation}
(t_e) \star [k_e, v_e] =  \Xi^{-1}_{\Gamma}\circ (t_e) \star \Xi_{\Gamma}[k_e, v_e] = \Xi^{-1}_{\Gamma} \circ([k_eh_e^{-1}, w_e][t_eh_e, v_e]) = [k_eh_e^{-1}t_eh_e, v_e]\\ 
\end{equation}

\noindent
Recall that $h_e$ is an element chosen so that $h_e \circ v_e = w_e = r_e\varrho_0$, for some $r_e \in \R_{\geq 0}$.  The action by $(t_e)$ is well-defined on any point $(k_e, v_e)$ with $|v_e| = |w_e| = r_e > 0$ by Proposition \ref{sl2xi}.

\subsection{Induced $\Gamma-$decomposition}

By Proposition \ref{sl2xi}, the map $\Xi_{\Gamma}$ is not $1-1$ on the subspace $M_{\Gamma}(SL_2(\C)) \setminus M_{\Gamma}^o(SL_2(\C))$, these are the points $(\ldots, [k_e, v_e], \ldots)$ with $|v_e| = 0$ for some subset of edges $S \subset E(\Gamma)$.  For any $S \subset E(\Gamma),$  define $\Gamma_S \subset \Gamma$ to be the graph induced by the $complement$ of $S$, and $\Gamma^S$ to be the graph induced by $S$ itself (see Figure \ref{gammas}). The next lemma addresses the structure of the possible $\Gamma_S$.

\begin{figure}[htbp]
\centering
\includegraphics[scale = 0.4]{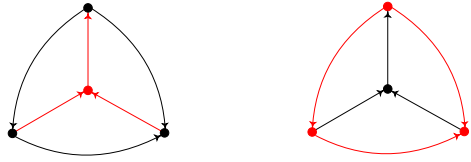}
\caption{The graphs $\Gamma^S$ (left) and $\Gamma_S$ (right)}
\label{gammas}
\end{figure}

\begin{lemma}\label{hyperplanesubgraph}
Let $S \subset E(\Gamma)$ be a set of edges for which some $(\ldots, [k_e, v_e], \ldots)$ has $v_e = 0$ if and only if $e \in S$, then
$\Gamma_S \subset \Gamma$ has no leaves.  For any subgraph $\Gamma' \subset \Gamma$ with no leaves, there exists
a $(\ldots, [k_e, v_e], \ldots) \in \mu_{SU(2)^{V(\Gamma)}}^{-1}(0)$ with $v_e = 0$ for precisely $e \in E(\Gamma) \setminus E(\Gamma')$.
\end{lemma}

\begin{proof}
Let $w \in V(\Gamma),$ have $\epsilon(w) = \{e_1, \ldots, e_m\} \cup \{f_1, \ldots, f_k\}$ with the $e_i$ incoming and $f_j$ outgoing.
For $(k_e, v_e) \in \mu_{SU(2)^{V(\Gamma)}}^{-1}(0)$ the momentum $0$ condition for the copy of $SU(2)$ associated to $w$ is the following. 

\begin{equation}
k_{e_1}\circ v_{e_1} + \ldots + k_{e_m}\circ v_{e_m} -( v_{f_1} + \ldots + v_{f_k}) = 0\\
\end{equation}

\noindent
If $\Gamma_S \subset \Gamma$ has a leaf, then for some $w \in V(\Gamma)$ all of the $v_{e_i}, v_{f_j}$ above 
are $0$ save one.  This contradicts the equation.   

For the second part, observe that if we can assign non-zero vectors to the edges of any graph which satisfy the momentum
$0$ condition above, we can take an such assignment to $\Gamma_S$ and extend it to $\Gamma$ by placing $0$ on all
edges of $S$.  For this reason, it suffices to show that any graph supports a non-zero vector assignment. 

We fix the genus $g$ and consider the graph $\Gamma_g$ with one vertex.  Any assignment of vectors $v_1, \ldots, v_g$ to the edges of $\Gamma_g$ satisfies the momentum $0$ condition, because $\sum Id\circ v_1 + \ldots + Id \circ v_g - (v_1 + \ldots + v_g) = 0.$ We take this as the base case of an induction on the number of edges in a graph $\Gamma.$  We let $\Gamma'$ be obtained from $\Gamma$ by collapsing an edge $f \in E(\Gamma)$, $\delta(f) = (x, y)$ to a vertex $w \in V(\Gamma')$, and we suppose there is some assignment of $(\ldots, [k_e, v_e], \ldots)$ to the edges of $\Gamma'$ which satisfies the momentum $0$ condition, with all $v_e \neq 0.$   The edges incident on $w$
are partitioned into two sets $\epsilon(w) = X \cup Y$ by $f$, with those in $X$ incident on $x$ and those $Y$ incident on $y$, with associated sums

\begin{equation}
\sigma = - [\sum_{i \in X} k_i \circ v_i - \sum_{j \in X} v_j] = \ \ \ \ [\sum_{i \in Y} h_i \circ z_i - \sum_{j \in Y} z_j]. \\
\end{equation}

There are two cases, $\sigma \neq 0$ and $\sigma = 0.$  If $\sigma \neq 0$, we place $(Id, \sigma)$ on the edge $f$.  As $f$ is incoming
on $y$, we have $-\sigma + [\sum_{i \in Y} h_i \circ z_i - \sum_{j \in Y} z_j] = -\sigma + \sigma = 0.$  We also have $ \sigma + [\sum_{i \in X} k_i \circ v_i - \sum_{j \in X} v_j] = \sigma - \sigma = 0,$ so this assignment satisfies the requisite conditions for $\Gamma.$   

If $\sigma = 0$, we have two more cases. If $f$ is a separating edge, we rotate all $[k_i, v_i]$, $i \in Y$ by some $g \in SU(2)$ to make the difference $\neq 0$, and reduce to the previous case.   If $f$ is not separating, we take it to be part of a simple loop $\gamma = (f, e_2, \ldots, e_k)$ in $\Gamma.$  Each edge of this loop is directed in some way, if two consecutive edges share a direction, we add a $q$ or $-q$ to both (correcting as necessary with $k_i^{-1}.$  If two consecutive edges switch directions, we switch signs.  This can always be done, because there must always be an even number of switch vertices.  To see this last point, note that if we have an odd number of switch vertices, then the edge orientation must switch an odd number of times as we go around $\gamma$, but this implies that $f$ has both possible orientations, a contradiction. 
\end{proof}

\begin{corollary}
The polyhedral cone $\mathcal{P}_{\Gamma} \subset \R^{E(\Gamma)}$ intersects the coordinate subspace $\R^K \subset \R^{E(\Gamma)}$
if and only if $K = E(\Gamma) \setminus E(\Gamma')$, for $\Gamma' \subset \Gamma$ a subgraph with no leaves. 
\end{corollary}

We let $M_{\Gamma, S}(SL_2(\C)) \subset M_{\Gamma}(SL_2(\C))$ and $M_{\Gamma, S}(SL_2(\C)^c) \subset M_{\Gamma}(SL_2(\C)^c)$ 
be the subspaces defined by the condition $|v_e| = 0$ if and only if $e \in S$. Clearly $\Xi_{\Gamma}$ restricts
to a surjective map $\Xi_{\Gamma}: M_{\Gamma, S}(SL_2(\C)) \to M_{\Gamma, S}(SL_2(\C)^c).$   The condition $|v_e| = 0$ if and only if $e \in S$  defines an open face of $\mathcal{P}_{\Gamma}$, which is the image of $M_{\Gamma, S}(SL_2(\C))$ and $M_{\Gamma, S}(SL_2(\C)^c)$ under the maps
$\mu_{\mathbb{T}}\circ \Xi_{\Gamma}$ and $\mu_{\mathbb{T}}$, respectively. Furthermore, $M_{\Gamma, S}(SL_2(\C)^c)$ is the $\mathbb{T}^{\C}$
orbit corresponding to this face.  For the next proposition, we allow a graph $\Gamma_S$ which may have bivalent vertices, the definition of $M_{\Gamma_S}(-)$ is identical. 

\begin{proposition}
The closure of $M_{\Gamma, S}(SL_2(\C)^c)$ is isomorphic to $M_{\Gamma_S}(SL_2(\C)^c)$. 
\end{proposition} 

\begin{proof}
By definition $M_{\Gamma, S}(SL_2(\C)^c)$ is covered by the space of $(\ldots, [k_e, v_e], \ldots)$ satisfying
the momentum $0$ condition $SU(2)^{V(\Gamma)}$ with $v_e = 0, k_e = Id$ for precisely those $e \in E(\Gamma_S).$ The closure of this space
admits points where $v_e = 0$, and therefore $k_e$ is identified to  $Id.$  This is the intersection of the subspace $(SL_2(\C)^c)^{E(\Gamma_S)} \subset (SL_2(\C)^c)^{E(\Gamma)}$ with the momentum level $0$ space.  As all $e \in S$ are assigned $0$ lengths, this subspace is 
precisely the momentum level $0$ set for the action of $SU(2)^{V(\Gamma_S)}$, and the condition $k_e =  Id$ implied by the definition of the contraction $SL_2(\C)^c$ implies the action of $SU(2)^{V(\Gamma)}$ reduces to that of $SU(2)^{V(\Gamma_S)}.$ 
\end{proof}

Let $\bar{\Gamma}_S$ be the graph obtained from $\Gamma_S$ by replacing any pair of edges connected to a bivalent vertex with a single edge.  We leave
it to the reader to check that $M_{\Gamma_S}(SL_2(\C)^c) \cong M_{\bar{\Gamma}_S}(SL_2(\C)^c)$.

The space $M_{\Gamma, S}(SL_2(\C))$ is an $SU(2)^{V(\Gamma)}$ quotient of the product space $SU(2)^S \times \mu_{SU(2)^{V(\Gamma_S)}}^{-1}(0)$.
Here $\mu_{SU(2)^{V(\Gamma_S)}}^{-1}(0) \subset SL_2(\C)^{E(\Gamma) \setminus S}$ is the level $0$ subset defined by the $SU(2)^{V(\Gamma_S)}$
action in the definition of $M_{\Gamma_S}(SL_2(\C)).$  In particular, if we take $S = E(\Gamma)$, the fiber over the unique point in $M_{\Gamma}(SL_2(\C)^c)$ defined by the condition $|v_e| =0$ for all $e \in E(\Gamma)$ is the character variety $M_{\Gamma}(SU(2)) = \mathcal{X}(F_g, SU(2)) \subset \mathcal{X}(F_g, SL_2(\C))$ associated to the group $SU(2).$ 

\begin{proposition}\label{connectedfibers}
The map $\Xi_{\Gamma}: M_{\Gamma, S}(SL_2(\C)) \to M_{\Gamma, S}(SL_2(\C)^c)$  has connected fibers. 
\end{proposition}

\begin{proof}
  The fiber over an equivalence class  $(\ldots, [k_e, v_e], \ldots) \in M_{\Gamma, S}(SL_2(\C)^c)$ is $SU(2)^S \times (\ldots, [k_e, v_e], \ldots)$, which is connected because the subspace of points in the equivalence class $(\ldots, [k_e, v_e], \ldots)$ is connected.  
\end{proof}

Now Proposition \ref{connectedfibers} and Lemma \ref{hyperplanesubgraph} complete a proof of Theorem \ref{maintoric}.

\section{Compactifications of $\mathcal{X}(F_g, SL_2(\C))$}\label{compactification}

In this section we show that the cone $C_{\Gamma, \phi}$ of valuations
on $\C[\mathcal{X}(F_g, SL_2(\C))]$ defined by a marking $\phi: \Gamma_g \to \Gamma$
is induced from a compactification $M_{\Gamma}(SL_2(\C)) \subset M_{\Gamma}(X).$
Here $X \subset \mathbb{P}^5$ is the projective $SL_2(\C) \times SL_2(\C)$ scheme from Section \ref{sl2} obtained by taking $Proj$ of the algebra $\bar{R} = \C[a, b, c, d, t]/<ad - bc - t^2>$
  We construct $M_{\Gamma}(X)$ as a $GIT$ quotient as in Section \ref{charactervaluations}, using the $SL_2(\C) \times SL_2(\C)$-linearized line bundle $\mathcal{O}(1)$ on $\mathbb{P}^4.$

\subsection{The $GIT$ quotient $M_{\Gamma}(X)$}

  As $X$ is an $SL_2(\C) \times SL_2(\C)$ variety,  $X^{E(\Gamma)}$ carries an action of the group $SL_2(\C)^{V(\Gamma)}$.  For each edge $e \in E(\Gamma)$ there is a corresponding divisor $\bar{D}_e = D \times \prod_{e \neq f \in E(\Gamma)} X$. Furthermore, there is a subspace $D_S = \cap_{e \in S} D_e = \prod_{e \in S} D \times \prod_{f \in E(\Gamma) \setminus S} X$ for every $S \subset E(\Gamma).$ Each subspace $D_S$ is $SL_2(\C)^{E(\Gamma)}$ stable, and smooth, making $\hat{D}_{\Gamma} = \cup D_e$ a normal crossings divisor on $X^{E(\Gamma)}$. 

We let $\mathcal{L} = i^*(\mathcal{O}(1))$ under $i: X \subset \mathbb{P}^4$, this bundle is $SL_2(\C) \times SL_2(\C)$-linearized and the section ring of $\mathcal{L}$ is then $\bar{R}$. Similarly we let $\mathcal{M} = j^*(O(1))$ for $j: D \subset \mathbb{P}^4$ the inclusion of the divisor $D \subset X$; the section ring of $\mathcal{M}$ is $\bar{R}/<t>$. We place $\mathcal{L}^{\boxtimes E(\Gamma)}$ on $X^{E(\Gamma)}$, which then comes with an $SL_2(\C)^{V(\Gamma)}$ action. We define $M_{\Gamma}(X)$ to be the following $GIT$ quotient.  

\begin{equation}
M_{\Gamma}(X) = SL_2(\C)^{V(\Gamma)} \ql_{\mathcal{L}^{\boxtimes E(\Gamma)}} X^{E(\Gamma)}\\
\end{equation}

The projective coordinate ring $\bar{R}_{\Gamma}$ associated to the line bundle $\mathcal{L}^{\boxtimes E(\Gamma)}$ is a direct sum of the following spaces. 

\begin{equation}
\bar{R}_{\Gamma}(m) = \bigoplus_{a: E(\Gamma) \to [0, m]} \bigotimes_{e \in E(\Gamma)} V(a(e)) \otimes V(a(e))t^m\\
\end{equation}

Once again, $t^m$ serves as a placeholder in this algebra, and multiplication by $t \in V(0, 0)^{\otimes E(\Gamma)}t \subset \bar{R}_{\Gamma}(1)$ shifts a summand in $\bar{R}_{\Gamma}(m)$ to its isomorphic image in $\bar{R}_{\Gamma}(m+1).$  The projective coordinate ring $R_{\Gamma} = $ $[\bar{R}_{\Gamma}]^{SL_2(\C)^{V(\Gamma)}}$ of $M_{\Gamma}(X)$ has a similar description to the affine coordinate ring of $\C[M_{\Gamma}(SL_2(\C))]$, it is a direct sum of the following spaces. 
 
\begin{equation}
R_{\Gamma}(m) = \bigoplus_{a: E(\Gamma) \to [0, m]} [\bigotimes_{e \in E(\Gamma)} V(a(e)) \otimes  V(a(e))]^{SL_2(\C)^{V(\Gamma)}}t^m\\
\end{equation}

\noindent
Each of the spaces in this sum is spanned by a single spin diagram $\Phi_a$, with $a(e) \leq m$ for all $e \in E(\Gamma).$ Multiplication on $R_{\Gamma}$ is induced from the inclusions $R_{\Gamma}(m) \subset \C[M_{\Gamma}(SL_2(\C))]$.  

\begin{proposition}
The space $M_{\Gamma}(SL_2(\C))$ is a dense, open subset of $M_{\Gamma}(X)$.
\end{proposition}

\begin{proof}
  The subspaces $R_{\Gamma}(m)$ filter the algebra $\C[M_{\Gamma}(SL_2(\C))]$, so the algebra $\frac{1}{t}R_{\Gamma}$ is isomorphic to $\C[M_{\Gamma}(SL_2(\C))]\otimes \C[t, \frac{1}{t}]$ by standard properties of Rees algebras. This implies that the complement of the $t = 0$ hypersurface is $M_{\Gamma}(SL_2(\C)).$
\end{proof}

\subsection{The boundary divisor $D_{\Gamma}$}

The divisor $D_{\Gamma} \subset M_{\Gamma}(X)$ is defined by the equation $t = 0$.  
The $m-$th graded piece of the ideal $<t >$ is the following subspace of $R_{\Gamma}$: 

\begin{equation}
<t> \cap R_{\Gamma}(m) = \bigoplus_{a: E(\Gamma) \to [0, m-1]}\C \Phi_a t^m.\\
\end{equation}

\noindent
The ideal $<t>$ is the intersection of the following subspaces $I_e$, $e \in E(\Gamma)$:

\begin{equation}
I_e(m) = \bigoplus_{a: E(\Gamma) \to [0, m], a(e) < m} \C \Phi_a t^m.
\end{equation}

\begin{proposition}
The subspaces $I_e$ are prime ideals, and therefore define the irreducible components of $D_{\Gamma}.$  
The locus $D_e$ of $I_e$ is the image in $M_{\Gamma}(X)$ of $\bar{D}_e \subset X^{E(\Gamma)}$, 
and can be identified with $SL_2(\C)^{V(\Gamma)} \ql_{\mathcal{M} \boxtimes \prod_{e' \in E(\Gamma) \setminus \{e\}} \mathcal{L}} [D \times \prod_{e' \in E(\Gamma) \setminus \{e\}} X].$
\end{proposition}

\begin{proof}
We let $\bar{I}_e \subset \bar{R}_{\Gamma}$ be the sum of the spaces:

\begin{equation}
\bar{I}_e(m) =  \bigoplus_{a: E(\Gamma) \to [0, m], a(e) < m}  \bigotimes_{e' \in E(\Gamma)} V(a(e')) \otimes V(a(e'))t^m \subset \bar{R}_{\Gamma}(m),\\
\end{equation}

\noindent
 so that  $I_e = \bar{I}_e \cap R_{\Gamma}$ and $I_e = \bar{I}_e^{SL_2(\C)^{V(\Gamma)}}$. By definition the locus of $\bar{I}_e$ is $\bar{D}_e = D \times \prod_{e' \in E(\Gamma) \setminus \{e\}} X \subset X^{E(\Gamma)}.$   
\end{proof}

The space $SL_2(\C)^{E(\Gamma)}$ is compactified by  $X^{E(\Gamma)}$, and the cone $C_{\Gamma}$  of valuations on $\C[SL_2(\C)^{E(\Gamma)}]$ are the divisorial valuations associated to the components $\bar{D}_e \subset \bar{D}_{\Gamma}.$  These valuations are induced on $\C[M_{\Gamma}(SL_2(\C))]$ by the inclusion $\C[M_{\Gamma}(SL_2(\C))] \subset \C[SL_2(\C)^{E(\Gamma)}]$, this allows us to  characterize the cone $C_{\Gamma}.$

\begin{proposition}
The divisorial valuations on $\C[M_{\Gamma}(SL_2(\C))]$ attached to the components $D_e$ are points on the extremal rays of $C_{\Gamma}.$ 
\end{proposition}

\begin{proof}
For an edge $e \in E(\Gamma)$ let $v_e$ be the valuation on $\C[M_{\Gamma}(SL_2(\C))]$ corresponding to the metric which assigns $e$ a $1$ and every other edge a $0$. By definition $v_e$ is the valuation induced on $\C[M_{\Gamma}(SL_2(\C))] \subset \C[SL_2(\C)^{E(\Gamma)}]$ by taking degree along $\bar{D}_e$. Let $\bar{\eta}_e$ be the generic point of $\bar{D}_e \subset X^{E(\Gamma)},$
and let $\eta_e$ be the generic point of $D_e \subset M_{\Gamma}(X).$  We have the inclusion of local rings $\mathcal{O}_{\eta_e} = \mathcal{O}_{\bar{\eta}_e}^{SL_2(\C)^{E(\Gamma)}} \subset \mathcal{O}_{\bar{\eta}_e}.$ Furthermore, the generator $t_e$ of the maximal ideal of $\mathcal{O}_{\bar{\eta}_e}$ is an invariant, it follows that this same element generates the maximal ideal of $\mathcal{O}_{\eta_e}$.
\end{proof}

For a set $S \subset E(\Gamma),$ the ideal $I_S = \sum_{e \in S} I_e$  cuts out $D_S = \cap_{e \in S} D_e:$

\begin{equation}
I_S(m) = \bigoplus_{a: E(\Gamma) \to [0, m], \exists e \in E(\Gamma), a(e) < m} \C \Phi_a t^m.\
\end{equation}

\begin{proposition}\label{codim}
Each $I_S \subset R_{\Gamma}$ is prime, and $codim(D_S) = |S|$. 
\end{proposition}

\begin{proof}
If $f = \sum C_i \Phi_{a_i} \in R_{\Gamma}(m)$, $g = \sum K_i \Phi_{b_i} \in R_{\Gamma}(n)$, $f, g \notin I_S$, then $a_i(e) = m$ and $b_i(e) = n$ for all $e \in S.$  Let $v_S = \sum_{e \in S} v_e: \C[M_{\Gamma}(SL_2(\C))] \to \Z \cup \{\infty\}$ be the sum of the extremal ray generators of $C_{\Gamma}$ corresponding to the edges in $S$.  We have $v_S(fg) = v_S(f) + v_S(g) = |S|(n+m),$ so it follows that some component $\Phi_c$ of $fg \in R_{\Gamma}(n+m)$ must have $c(e) = n + m$ for all $e \in S.$  

For the second part we show that for any $S \subset S'$ with $|S' \setminus S| = 1$ there is an element $\Phi_a \in I_{S'}$, $\notin I_S.$  This implies that the height of $I_S$ is $|S|$, as it can placed in a strict chain of prime ideals of length $|E(\Gamma)| = dim(M_{\Gamma}(SL_2(\C)))$. The element $\Phi_a$ must have an $a: E(\Gamma) \to [0, m]$ with $a(e) = m$ for all $e \in S$ and $a(f) < m$ for $f \in S' \setminus S$, for some $m \in \Z.$  We let $m = 4$, and observe that each of the trinodes in Figure \ref{strattri}, and all of their permutations, can be part of a spin diagram.  These can be combined without limitations to produce a weighting with the desired properties for any $\Gamma.$

\begin{figure}[htbp]
\centering
\includegraphics[scale = 0.4]{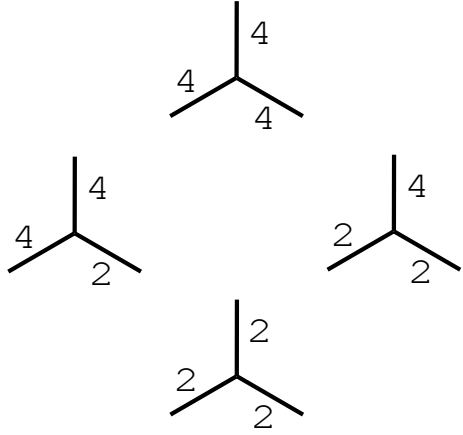}
\caption{building blocks of the element $a \in I_{S'}, \notin I_S$}
\label{strattri}
\end{figure}

\end{proof}

\subsection{The induced toric degeneration of $M_{\Gamma}(X)$}

The projective coordinate ring $R_{\Gamma}$ is a subring of $\C[M_{\Gamma}(SL_2(\C)]] \otimes \C[t],$
and inherits the basis of spin diagrams $\Phi_at^m$.  Lemma \ref{val1} implies that the valuations $v_{\Gamma, \ell} \in C_{\Gamma}$
all extend to $\C[M_{\Gamma}(SL_2(\C)]] \otimes \C[t]$ and $R_{\Gamma}$.  We use an interior element of $C_{\Gamma}$ to induce
a toric degeneration on $M_{\Gamma}(X).$  The following scheme is the special fiber of this degeneration, in analogy with $M_{\Gamma}(SL_2(\C)^c).$  Recall the flat degeneration $i: X_0 \subset \mathbb{P}^4$ of $X$ from Section \ref{sl2} and let $\mathcal{L}_0 = i^*(O(1))$. 

\begin{definition}
Let $M_{\Gamma}(X_0)$ be the $GIT$ quotient of $X_0^{E(\Gamma)}$ with respect to the
$SL_2(\C)^{V(\Gamma)}-$linearized line bundle $\mathcal{L}_0^{\boxtimes E(\Gamma)}.$ 
\end{definition}

\begin{proposition}
An interior point $v_{\Gamma, \ell} \in C_{\Gamma}$ defines a flat degeneration $M_{\Gamma}(X) \Rightarrow M_{\Gamma}(X_0)$. 
\end{proposition}

\begin{proof}
This argument is identical to the proof of Proposition \ref{agraded}.
\end{proof} 

We let $\bar{T}_{\Gamma}$ be the graded coordinate ring of $X_0^{E(\Gamma)}$ with respect to the line bundle $\mathcal{L}_0^{\boxtimes E(\Gamma)}$.  The invariant subring $T_{\Gamma} \subset \bar{T}_{\Gamma}$ has an identical description to $R_{\Gamma}$ as a direct
sum of spaces associated to spin diagrams $a: E(\Gamma) \to Z$. 

\begin{equation}
T_{\Gamma}(m) = \bigoplus_{a: E(\Gamma) \to [0, m]} [\bigotimes_{e \in E(\Gamma)} V(a(e)) \otimes V(a(e))]^{SL_2(\C)^{V(\Gamma)}} t^m\\
\end{equation}

The multiplication operation on this algebra is induced from $\bar{T}_{\Gamma}$, so it follows that
the product of elements from components labeled by $a, b: E(\Gamma) \to \Z$ is an element
in the component labeled by $a + b: E(\Gamma) \to \Z.$  As each of these components is multiplicity
free, it follows that $T_{\Gamma}$ is a graded (by the exponent of $t$)  affine semigroup algebra.  

\begin{definition}
We define the polytope $\mathcal{Q}_{\Gamma}$ to be the subset of $a \in \mathcal{P}_{\Gamma}$ satisfying $a(e) \leq 1.$
\end{definition}

\begin{proposition}
The algebra $T_{\Gamma}$ is isomorphic to the graded affine semigroup algebra
defined by $\mathcal{Q}_{\Gamma}$ with respect to the lattice $L_{\Gamma}.$
\end{proposition}

\begin{proof}
The description of $T_{\Gamma}(m)$ above implies that
the basis of this space defined by the direct sum decomposition is in bijection with the lattice points of $m\mathcal{Q}_{\Gamma}.$ 
Multiplication in this algebra is computed by addition on the direct sum labels $a: E(\Gamma) \to \Z, m \in \Z.$
\end{proof}

\begin{example}[Graphs of genus $2$]
The character variety $\mathcal{X}(F_2, SL_2(\C))$ is isomorphic to $\C^3$.  We depict
the compactification polytopes $\mathcal{Q}_{\Gamma_1}, \mathcal{Q}_{\Gamma_2}$ for the two graphs of genus $2$
in Figure \ref{2polytopes}.  We have colored the intersection of $\mathcal{Q}_{\Gamma_1}, \mathcal{Q}_{\Gamma_2}$ with
the coordinate hyperplanes in red.  The fibers of $\Xi_{\Gamma_i}$ for points in these intersections
are products of $3-$spheres $\mathbb{S}^3 \cong SU(2)$, with the exception of the fiber over
the origin, which is $\mathcal{X}(F_2, SU(2)).$

\begin{figure}[htbp]
\centering
\includegraphics[scale = 0.4]{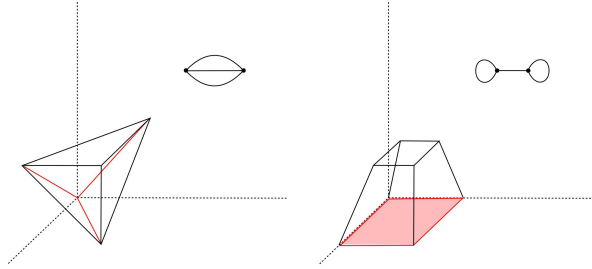}
\caption{The polytopes $\mathcal{Q}_{\Gamma_1}, \mathcal{Q}_{\Gamma_2}$ for graphs of genus $2$}
\label{2polytopes}
\end{figure}

\end{example}

The conditions which define $I_S \subset R_{\Gamma}$ likewise define ideals $J_S \subset T_{\Gamma}$. We let $K_S = \cap K_e$, for $K_e$ the divisor defined by $J_e.$ 

\begin{proposition}

\begin{enumerate}
\item The ideal $J_S$ cuts out the toric subspace $K_S \subset M_{\Gamma}(X_0)$ defined
by the face $F_S = \{w \in \mathcal{Q}_{\Gamma}, e \in S | w(e) = 1\}$.\\
\item $codim(F_S) = codim(K_S) = |S|$.\\
\end{enumerate}
\end{proposition}

\begin{proof}
The first part follows from the observation that $S_{\Gamma}^{SL_2(\C)^{V(\Gamma)}}(m)/J_S(m)$
has a basis in bijection with the lattice points of $mF_S.$  The second part is identical to proof of 
Proposition \ref{codim}. 
\end{proof}

\noindent
Finally we relate $D_S$ and $K_S.$

\begin{proposition}
The space $K_S$ is a toric degeneration of $D_S.$
\end{proposition}

\begin{proof}
The graded component $R_{\Gamma}(m)/I_S(m)$ of the coordinate ring of has a basis $[\Phi_a]$ for $[a] \in mF_S \subset m\mathcal{Q}_{\Gamma}.$
For $[\Phi_a] \in R_{\Gamma}(m)/I_S(m),$ $[\Phi_b] \in R_{\Gamma}(n)/I_S(n),$  multiplication $[\Phi_a][\Phi_b]$ gives a sum $C_i[\Phi_{c_i}]$ for $[c_i] \in (n +m)F_S$, and $c_1 = a + b$.  This expansion is obtained from the one in $R_{\Gamma}$ by eliminating those $\Phi_{c_i}$ with $c_i \in I_S.$  This is implies that we may use any filtration induced by a $v_{\ell}$ in the interior of $C_{\Gamma}$ can be used to obtain the toric degeneration to $K_S.$ 
\end{proof}

This concludes the proof of Theorem \ref{maincompact}.

\subsection{Newton-Okounkov bodies of $\mathcal{X}(F_g, SL_2(\C))$}\label{NOKsection}

We fix a graph $\Gamma$, an isomorphism $\gamma: \pi_1(\Gamma) \cong F_g$,
and we let $v_{\Gamma, \gamma, \vec{1}}$ be the valuation on $\C[M_{\Gamma}(SL_2(\C)]]$
associated to the metric which assigns every edges $e \in E(\Gamma)$ length $1.$
The toric degeneration $\C[P_{\Gamma}]$ matches the associated
graded ring of a Newton-Okounkov body construction explored
in \cite{M14}.  We show that that the maximal rank 
valuation used in this construction can be built from a flag of subvarieties
in the boundary divisor $D_{\Gamma} \subset M_{\Gamma}(X).$

Let $X$ be a projective variety of dimension $d$, and let $\vec{F}$ be a full flag of irreducible subvarieties, with
$F_1 = \{pt\}$ a smooth point of $X.$  Following \cite{KK}, one obtains a maximal rank valuation $\mathfrak{v}_{\vec{F}}$ on the
rational functions $K(X)$ as follows.  For $f \in K(X)$, one takes the degree $v_{F_d}(f)$ along the divisor $F_d$.  For $y_d \in \mathcal{O}_{pt}$ a local equation which defining $F_d$, one then repeats this process on $y_d^{- v_{F_d}(f)}f$, thought of as a
regular function on $F_d$, with respect to the flag $F_1 \subset \ldots F_{d-1}.$  The resulting function from $K(X)$ to $\Z^d$ (considered
with the lexicographic ordering) defines the associated valuation on $K(X).$ 

A total ordering $\prec$ on $E(\Gamma)$ defines an associated  lexicographic ordering on the spin diagrams $\Phi_a$ (\cite[Theorem 1.1]{M14}). We say $a < b$ if this is the case in the lexicographic ordering defined on the entries $a(e), b(e)$, $e \in E(\Gamma)$ with respect to the lexicographic ordering defined by $\prec$.   This construction defines a filtration $\mathfrak{v}_{\Gamma, \prec}$ on $\C[M_{\Gamma}(SL_2(\C))]$ which yields the Newton-Okounkov body construction in \cite{M14}.  The image of $\mathcal{R}(\Gamma, \phi)$ under $\mathfrak{v}_{\Gamma, \prec}$ is shown to give a basis, and it follows by Proposition \ref{gammainitial} that the same is true for $\mathcal{S}(\Gamma, \gamma).$ There is an associated ordering on the irreducible components $D_{e} \subset D_{\Gamma}$, we let $D_i$ be the intersection of the first $i$ of these components under $\prec,$ this defines a complete flag $\vec{D}_{\prec}$ of subspaces of $M_{\Gamma}(X)$.

\begin{theorem}\label{NOK}
Let $\prec$ be an ordering on the edges $E(\Gamma)$, and
let $\mathfrak{v}_{\Gamma, \prec}$ be the associated maximal rank valuation. 
Then $\mathfrak{v}_{\Gamma, \prec}$ is the maximal rank valuation defined
by the flag $\vec{D}_{\prec}.$
\end{theorem}

\begin{proof}

Everything is $SL_2(\C)^{V(\Gamma)}$ invariant, so we carry out our analysis in the scheme $X^{E(\Gamma)}$. Note that the
scheme $D_i$ is the image of the subspace $D^i \times X^{|E(\Gamma)| -i} \subset X^{E(\Gamma)}$.

Let $\tau_w \in \mathcal{S}(\Gamma, \gamma)$ be the trace-word function
for a reduced word $w \in F_g$ with associated $\Gamma$ tensor $\mathcal{V}(P, \phi).$ By the description of $X$ in Subsection \ref{wonderful}, 
$\tau_w$ can be written as a polynomial in $\C[SL_2(\C)]$ generators $A_e, B_e, C_e, D_e$ for $e \in E(\Gamma)$.
We let these be represented by $a_e/t_e, b_e/t_e, c_e/t_e,$ and $d_e/t_e$ for $t_e$ the form which cuts out $D \subset X$, note that the $t_e$ are algebraically independent.  We claim that some monomial in this polynomial has $t_e$ power equal to $-a(P, \phi)(e)$
for all $e \in E(\Gamma)$.  It follows from the definition of $\tau_w$ as a regular function on $M_{\Gamma}(SL_2(\C))$ that each monomial
has $t_e$ power $\geq -a(P, \phi)(e)$. Furthermore, if no monomial has all minimal possible $t_e$ powers then $v_{\Gamma, \gamma, \vec{1}}(\tau_w) < \sum_{e \in E(\Gamma)} a(P, \phi)(e)$, where $v_{\Gamma, \gamma, \vec{1}}$ is considered as a valuation on the rational functions of $X^{E(\Gamma)}.$ 

Following the Newton-Okounkov body recipe, we record the value $a(P, \phi)(e_1)$ and consider the product $t_{e_1}^{a(P, \phi)(e_1)}\tau_w$ as a regular function on $D \times X^{|E(\Gamma)| -1}$.  But by the above argument, the leading term of this product
has $t_{e_2}$ degree $\geq - a(P, \phi)(e_2)$, and this value is achieved on one of the same monomials with $t_{e_1}$ power $t_{e_1}^{-a(P, \phi)(e_1)}$.  Furthermore this pattern must continue as we proceed deeper into the flag. We conclude that $\mathfrak{v}_{\Gamma, \prec}(\tau_w) = (a(P, \phi)(e_1), \ldots, a(P, \phi)(e_{|E(\Gamma)|})),$ where the ordering is determined by $\prec.$
\end{proof}

\bibliographystyle{alpha}
\bibliography{Biblio}

\bigskip
\noindent
Christopher Manon:\\
cmanon@gmu.edu\\
Department of Mathematics,\\ 
George Mason University\\ 
Fairfax, VA 22030 USA 

\end{document}